\documentclass{amsart}

\usepackage{kotex}
\usepackage{amsmath}
\usepackage{amsfonts}
\usepackage{amscd}

\usepackage{amssymb}
\usepackage{graphicx}
\usepackage{caption}
\usepackage{subcaption}
\usepackage{dsfont}
\usepackage{color}
\usepackage{hyperref}
\usepackage{bbm}
\usepackage{a4wide}
\usepackage{sseq}
\usepackage{tikz-cd}
\usepackage[abs]{overpic} 

\newtheorem{theorem}{Theorem}[section]
\newtheorem{lemma}[theorem]{Lemma}
\newtheorem{proposition}[theorem]{Proposition}

\theoremstyle{definition}
\newtheorem*{definition*}{Definition}
\newtheorem{definition}[theorem]{Definition}

\theoremstyle{remark}
\newtheorem*{remark*}{Remark}
\newtheorem{remark}[theorem]{Remark}

\numberwithin{equation}{section}

\newcommand{\myproof}[2]{Proof of {#1} {#2}}

\begin{document}

\title[Berezin-Toeplitz Quantization in Real Polarizations with Toric Singularities]{Berezin-Toeplitz Quantization in Real Polarizations with Toric Singularities}

\author{NaiChung Conan Leung, AND YuTung Yau}
\address{The Institute of Mathematical Sciences and Department of Mathematics, The Chinese University of Hong Kong, Shatin, Hong Kong}
\email{leung@math.cuhk.edu.hk}
\address{The Institute of Mathematical Sciences and Department of Mathematics, The Chinese University of Hong Kong, Shatin, Hong Kong}
\email{ytyau@math.cuhk.edu.hk}

\thanks{}

\maketitle

\begin{abstract}
	On a compact K\"ahler manifold $X$, Toeplitz operators determine a deformation quantization $(\operatorname{C}^\infty(X, \mathbb{C})[[\hbar]], \star)$ with separation of variables \cite{K1996} with respect to transversal complex polarizations $T^{1, 0}X, T^{0, 1}X$ as $\hbar \to 0^+$ \cite{S2000}. The analogous statement is proved for compact symplectic manifolds with transversal non-singular real polarizations \cite{LY2021}.\par
	In this paper, we establish the analogous result for transversal \emph{singular} real polarizations on compact toric symplectic manifolds $X$. Due to toric singularities, half-form correction and localization of our Toeplitz operators are essential. Via norm estimations, we show that these Toeplitz operators determine a star product on $X$ as $\hbar \to 0^+$.
\end{abstract}

\section{Introduction}
\emph{Berezin-Toeplitz (BT) quantization} \cite{BMS1994, KS2000, S2000} is a quantization scheme on a compact K\"ahler manifold $(X, \omega)$, where $[\omega] = c_1(L)$ with $L$ a prequantum line bundle, bridging deformation quantization and geometric quantization. The \emph{Toeplitz operators} $Q^\hbar: \mathcal{C}^\infty(X, \mathbb{C}) \to \operatorname{End}_\mathbb{C} H^0(X, L^{\otimes k})$ completely determine a deformation quantization $(\mathcal{C}^\infty(X, \mathbb{C})[[\hbar]], \star^{\operatorname{BT}})$, which acts on $H^0(X, L^{\otimes k})$ asymptotically as $\hbar = \tfrac{1}{k} \mapsto 0^+$ \footnote{Throughout this paper, apart from the discussion on deformation quantization, we always let $\hbar = \tfrac{1}{k}$.}. In fact, this star product is particularly nice, namely it is with separation of variables \cite{K1996} and it can be obtained directly using BV quantization using the transversal complex polarizations $T^{1, 0}X$ and $T^{0, 1}X$ \cite{CLL2020a, CLL2020b, CLL2020c, CLL2021}.\par
In \cite{LY2021}, the authors established corresponding results for real polarizations where analysis involving distributions are needed. $T^{1, 0}X$ and $T^{0, 1}X$ for any complex polarization are now replaced by transversal real polarizations. Compact symplectic manifolds admitting transversal real polarizations, called \emph{para-Kahler manifolds} \cite{EST2006}, are essentially symplectic tori up to finite covers.\par
In this paper, we study BT quantization in \emph{singular} real polarizations on a compact toric symplectic manifold $X$. Let
\begin{equation*}
	\mu: X \to \mu(X) = P_X
\end{equation*}
be a moment map for the action of $T = \mathbb{R}^n/(2\pi \mathbb{Z})^n$. It determines a real polarization $\mathcal{P}^{\operatorname{v}}$, singular along the union $D$ of toric divisors in $X$. On $\mathring{X} = X \backslash D \cong (\mathbb{C}^*)^n$, we could choose action-angle coordinates $(x, \theta)$, thus the theta coordinates induce another singular real polarization $\mathcal{P}^{\operatorname{h}}$ on $X$ transversal to $\mathcal{P}^{\operatorname{v}}$. Assuming $[\tfrac{\omega}{2\pi}] - \tfrac{1}{2} c_1(X) \in H^2(X, \mathbb{Z})$, we obtain the half-form corrected quantum Hilbert space $\mathcal{H}^\hbar$ in polarization $\mathcal{P}^{\operatorname{v}}$ for each odd $k \in \mathbb{N}$. Using transversal polarizations $\mathcal{P}^{\operatorname{h}}, \mathcal{P}^{\operatorname{v}}$, we define a star product $\star$ with separation of variables on $(X, \omega)$, which is singular on $D$, and construct a Toeplitz-type operator $Q^\hbar: \mathcal{C}^\infty(X, \mathbb{C}) \to \operatorname{End}_\mathbb{C} \mathcal{H}^\hbar$. Our main result is that $\star$ is determined by $Q^\hbar$ with $\hbar \to 0^+$. Naively we would like to have the norm estimation
\begin{equation}
\label{Equation 2.13}
\lVert E_N^\hbar \rVert = \operatorname{O}(\hbar^{N+1}),
\end{equation}
for the error term $E_N^\hbar = Q_f^\hbar \circ Q_g^\hbar - Q_{f \star_N g}^\hbar$ , where $f \star_N g$ is the $N$th order truncation of the power series $f \star g$ in $\hbar$. Due to singularity, however, we need to consider the Poisson subalgebra $\mathcal{C}^\infty(X, \mathbb{C})_{<\infty}$ of functions $f \in \mathcal{C}^\infty(X, \mathbb{C})$ having finitely many non-zero fibrewise Fourier coefficients and the operator norm $\left\lVert E_N^\hbar \right\rVert_V$ of the restriction $E_N^\hbar \vert_{\mathcal{H}_V^\hbar}: \mathcal{H}_V^\hbar \to \mathcal{H}^\hbar$, where $\mathcal{H}_V^\hbar$ is the subspace of quantum states in $\mathcal{H}^\hbar$ supported on a relatively compact open subset $V$ of the interior $\mathring{P}_X$ of $P_X$. Our main result is as follows.
\begin{theorem}
	\label{Theorem 1.1}
	Let $f \in \mathcal{C}^\infty(X, \mathbb{C})$, $g \in \mathcal{C}^\infty(X, \mathbb{C})_{< \infty}$ and $V$ be a relatively compact open subset of $\mathring{P}_X$. Then there exists $\delta > 0$ such that for all $N \in \mathbb{N} \cup \{0\}$, there exists $K > 0$ such that
	\begin{equation}
	\left\lVert Q_f^\hbar \circ Q_g^\hbar - Q_{f \star_N g}^\hbar \right\rVert_V \leq K \hbar^{N+1},
	\end{equation}
	for all odd $k \in \mathbb{N}$ with $\hbar = \tfrac{1}{k} < \delta$.
\end{theorem}
% A relatively compact subset (or precompact subset) of a topological space is a subset whose closure is compact.

This theorem implies that the Toeplitz operators $Q^\hbar$ determine the singular star product $\star$ as $\hbar \to 0^+$.\par
This theorem is a counterpart of Theorem 1.1 in \cite{LY2021}. Nonetheless, there are several subtleties due to the toric singularities. First, the half-form correction \cite{KMN2013} plays a crucial role in the construction of $Q^\hbar$ by prohibiting the existence of degenerate Bohr-Sommerfeld fibers. This avoids problems caused by singularity of the star product $\star$ on degenerate $T$-orbits of $X$. Adopting this correction, we need to assume $[\tfrac{\omega}{2\pi}] - \tfrac{1}{2} c_1(X) \in H^2(X, \mathbb{Z})$. Then $[\tfrac{\omega}{2\pi}] - \tfrac{1}{2} c_1(X)$ is the first Chern class of a line bundle $L_+$, which is naively regarded as $L \otimes \sqrt{K}$, although $L$ and $\sqrt{K}$ might not be well-defined line bundles. Indeed, $L$ is a line bundle if and only if $X$ is spin. In any case, the `square' $L^{\otimes 2}$ is a line bundle and thus for $k$ odd, we have a line bundle $L_+^k$, which is regarded as $L^{\otimes k} \otimes \sqrt{K}$, whose first Chern class is $k[\tfrac{\omega}{2\pi}] - \tfrac{1}{2} c_1(X)$, serving as the (level-$k$) half-form corrected prequantum line bundle. The corresponding quantum Hilbert space $\mathcal{H}^\hbar$ in polarization $\mathcal{P}^{\operatorname{v}}$ has a Bohr-Sommerfeld bases $\{\sigma_\hbar^m\}_{m \in \Lambda_\hbar}$ for which $\sigma_\hbar^m$ is supported on the level-$k$ Bohr-Sommerfeld fibre over the point $x = \hbar m$ with $m \in \Lambda_\hbar$, where $\Lambda_\hbar$ is the set of lattice points in the \emph{corrected polytope} in level $k$.\par
Second, similar to \cite{LY2021}, the proof of our main theorem requires estimations on the fibrewise Fourier coefficients of functions and their derivatives on $X$. However, we need to be careful what is meant by fibrewise Fourier transform as $X$ has degenerate $T$-orbits. This was discussed in \cite{C2007}, which studied the groupoid approach to quantization for toric manifolds, and the definition of fibrewise Fourier transform depends on the choice of the action-angle coordinates $(x, \theta)$. We shall see from Proposition \ref{Proposition 3.9} that the Toeplitz operator $Q_f^\hbar$ for $f \in \mathcal{C}^\infty(X, \mathbb{C})$ is a linear combination of multiplication by the fibrewise Fourier coefficients $\widehat{f}_p$ of $f$ followed by the respective \emph{truncated} shift operators by $\hbar p$ (for $p \in \mathbb{Z}^n$). For instance, take $X = \mathbb{C}\mathbb{P}^1$ and suppose $f$ is of the form
\begin{align*}
	f(x, \theta) = \widehat{f}_p(x) e^{\sqrt{-1} p \theta} \quad \text{on } \mathring{X},
\end{align*}
with $p \in \mathbb{Z}$, i.e. all fibrewise Fourier coefficients of $f$ but $\widehat{f}_p$ vanish everywhere. Then for $m \in \Lambda_\hbar$,
\begin{equation*}
	Q_f^\hbar \sigma_\hbar^m = \begin{cases}
	\widehat{f}_p(\hbar m) \sigma_\hbar^{m+p}, & \quad \text{if } m + p \in \Lambda_\hbar;\\
	0, & \quad \text{otherwsie.}
	\end{cases}
\end{equation*}\par
Third, from the above observation that $Q^\hbar$ involves shift operators which are truncated, we have to cut off quantum states too close to the boundary of $P_X$ by taking a relatively compact open subset $V$ of $\mathring{P}_X$, and control fibrewise Fourier coefficients $\widehat{g}_q$ of the function $g$ appeared in the error term $E_N^\hbar$ for $q$ large by restricting $g$ to be in $\mathcal{C}^\infty(X, \mathbb{C})_{<\infty}$.\par
As a remark, Kirwin, Mour\~ao and Nunes \cite{KMN2013} showed that under the degeneration of K\"ahler polarizations to the singular real polarization $\mathcal{P}^{\operatorname{v}}$, the quantum Hilbert space $\mathcal{H}^\hbar$ in polarization $\mathcal{P}^{\operatorname{v}}$ is the limit of the quantum Hilbert spaces in K\"ahler polarizations. However, the Berezin-Toeplitz star products and Toeplitz operators do not converge. The reason is that both $T^{1, 0}X$ and $T^{0, 1}X$, which are transversal complex polarizations, converge to the same polarization $\mathcal{P}^{\operatorname{v}}$. Therefore, we lose a pair of transversal polarizations under K\"ahler degeneration, and our Toeplitz operators are not the limit of Toeplitz operators for K\"ahler polarizations.\par
The paper is organized as follows. In Section \ref{Section 2}, we review toric geometry which is relevant to our construction of Toeplitz-type operators. In Sections \ref{Section 3} and \ref{Section 4}, we discuss deformation quantizations compatible with singular real polarizations and half-form corrected geometric quantization of $X$ by the work \cite{KMN2013} respectively. In Section \ref{Section 5}, we provide a proof of our main theorem.

%\cite{S1977}, \cite{BHSS1991}, \cite{D2001}, \cite{T1987}, \cite{T2016}

% See also `Short Star-Products for Filtered Quantizations, I', Pavel Etingof, Douglas Stryker, 2019.
% https://arxiv.org/pdf/1909.13588.pdf

%\textcolor{red}{locality?}

\subsection{Acknowledgement}
\quad\par
%\indent
% Conan said thanks Kwok Wai Chan, Qin Li, Si Li, Zhiming Ma.
We thank Kwok Wai Chan, Qin Li, Si Li and Ziming Nikolas Ma for useful comments as well as helpful discussions. This research was substantially supported by grants from the Research Grants Council of the Hong Kong Special Administrative Region, China (Project No. CUHK14301619 and CUHK14306720) and direct grants from the Chinese University of Hong Kong.
% Y. T. Yau
% The second author would like to thank Yat-Hin Suen, Dan Wang and Ziming Nikolas Ma for helpful discussions on geometric quantization in real toric polarizations, and thank Han Cui and Man-Chun Lee for assisting him to figure out the techniques in Fourier analysis that is useful for the current problem.

\section{Preliminaries on Compact Toric Symplectic Manifolds}
\label{Section 2}
In this section, we set up notations on a compact toric symplectic manifold and review certain aspects on them that are relevant to our construction of Toeplitz-type operators in Section \ref{Section 5}, including action-angle coordinates, polytopes and $T$-equivariant Hermitian line bundles, fibrewise Fourier transform and (singular) real polarizations.\par % key ingredients
Our notations mainly follow from \cite{KMN2013}. Consider a compact toric symplectic manifold $(X, \omega)$ with a toric action of $T = \mathbb{R}^n/(2\pi \mathbb{Z})^n$ and a moment map
\begin{align*}
	\mu: X \to \mathfrak{t}^* \cong \mathbb{R}^n,
\end{align*}
where $\mathfrak{t}$ is the Lie algebra of $T$. Let $\Sigma$ be the normal fan associated to the moment polytope $P_X := \mu(X)$ of $(X, \omega)$ and $\Sigma^{(1)}$ be the set of rays in $\Sigma$. We have % $l_F(x) \geq 0$, $F \in \Sigma^{(1)}$
\begin{align*}
	% P_X = \bigcap_{F \in \Sigma^{(1)}} l_F^{-1}([0, +\infty))
	P_X = \{ x \in \mathfrak{t}^*: \text{for all } F \in \Sigma^{(1)}, l_F(x) \geq 0 \},
\end{align*}
where $l_F(x) := \langle x, \nu_F \rangle + \lambda_F$, $\nu_F$ is the integral inward pointing vector normal to the facet $F$ of $P_X$ and $\lambda_F \in \mathbb{R}$.
% Note that \textcolor{red}{we have choices of $\lambda_F$} up to translations of the moment polytope $P_X$.

\subsection{Action-angle coordinates}
\quad\par
On $(X, \omega)$ there is a set of action-angle coordinates useful for later discussions, which is indexed by vertices of $P_X$ and a distinguished element $\circ$ indicating a chart on $\mathring{X} := \mu^{-1}(\mathring{P}_X)$, where $\mathring{P}_X$ is the interior of $P_X$. Let $\operatorname{Vert}(P_X)$ be the set of vertices of $P_X$ and $\overline{\operatorname{Vert}}(P_X) = \operatorname{Vert}(P_X) \sqcup \{\circ\}$.
\begin{itemize}
	\item On $\mathring{X} \cong \mathring{P}_X \times T \cong (\mathbb{C}^*)^n$, we have action-angle coordinates $(x, \theta)$, so that $\mu(x, \theta) = x$ and $\omega \vert_{\mu^{-1}(\mathring{P}_X)} = \sum_{i=1}^n dx^i \wedge d\theta^i$. We also write $(x_\circ, \theta_\circ)$ for $(x, \theta)$ and $U_\circ$ for $\mathring{X}$.
	\item Suppose $v \in \operatorname{Vert}(P_X)$. There are exactly $n$ facets of $P_X$ adjacent to $v$ and we order them as $F_{v, 1}, ..., F_{v, n} \in \Sigma^{(1)}$. Define new coordinates $x_v$ on $\mathfrak{t}^*$ and $\theta_v$ on $T$ by
	\begin{equation}
	\label{Equation 3.1}
	x_v := A_v x + \lambda_v, \quad \theta_v := {}^{\operatorname{t}}A_v^{-1} \theta,
	\end{equation}
	where $A_v = {}^{\operatorname{t}}\begin{pmatrix}
	\nu_{F_{v, 1}} & \cdots \nu_{F_{v, n}}
	\end{pmatrix} \in \operatorname{GL}(n, \mathbb{Z})$ and $\lambda_v = {}^{\operatorname{t}}(\lambda_{F_{v, 1}}, ..., \lambda_{F_{v, n}})$. In addition, define
	\begin{equation*}
		U_v := \mu^{-1} \left( \{v\} \cup \bigcup_{F: \text{faces of } P_X \text{adjacent to } v } \mathring{F} \right),
	\end{equation*}
	where $\mathring{F}$ denotes the interior of a face $F$ of $P_X$ (By convention, $\mathring{\{v\}} = \{v\}$). We call $(U_v, x_v, \theta_v)$ a \emph{vertex coordinate chart}. On $U_v$, the symplectic form can be written as $\omega \vert_{U_v} = \sum_{i=1}^n dx_v^i \wedge d\theta_v^i$, but note that $(x_v, \theta_v)$ are singular along the faces of $P_X$ adjacent to $v$ with positive codimension in the same way that polar coordinates are singular at the origin. Indeed, one can obtain nonsingular coordinates $(a_v, b_v)$ on $U_v$ via the transformations:
	\begin{equation*}
		a_v^i := \sqrt{x_v^i} \cos \theta_v^i, \quad b_v^i := \sqrt{x_v^i} \sin \theta_v^i, \quad \text{for all } i \in \{1, ..., n\}.
	\end{equation*}
\end{itemize}

\subsection{Polytopes and $T$-equivariant Hermitian line bundles}
\quad\par
\label{Subsection 2.2}
In this subsection, we shall review the correspondence between $T$-equivariant ample line bundles on $X$ and integral polytopes with normal fan $\Sigma$. For more details, see Subsection 2.2.3 in \cite{KMN2013}.\par
% It is known that, with respect to the $T$-invariant complex structure $J$ determined by any symplectic potential for $(X, \omega)$, the linear equivalence classes of the $T$-invariant divisors of $X$ generate the Picard group of $X$, and there is a one-to-one correspondence between irreducible $T$-invariant divisors and $1$-cones in $\Sigma$. It means that if we specify suitable data for $1$-cones in $\Sigma$, then we have a corresponding holomorphic line bundle $L$. It turns out in Subsection 2.2.3 in \cite{KMN2013} that one obtain a $T$-equivariant Hermitian line bundle $L^{\operatorname{U}(1)}$ out of $L$, and $L^{\operatorname{U}(1)}$ has transition functions independent of the choice of $J$. Since the choice of $T$-invariant complex structures is basically irrelevant to the construction of Toeplitz-type operators in Section \ref{Section 5},  we avoid mentioning it explicitly and describe directly how a set of combinatorial data on $1$-cones in $\Sigma$ yields a $T$-equivariant Hermitian line bundle.\par
Consider a map $\Sigma^{(1)} \to \mathbb{R}$ given by $F \mapsto \lambda_F^L$. There associates the following polytope
\begin{align*}
P_L := \{ x \in \mathfrak{t}^*: \text{for all } F \in \Sigma^{(1)}, l_F^L(x) \geq 0 \},
\end{align*}
where $l_F^L(x) := \langle x, \nu_F \rangle + \lambda_F^L$ (for instance, the polytope associated to the map $\Sigma^{(1)} \to \mathbb{R}$ given by $F \mapsto \lambda_F$ is the moment polytope $P_X$ for the moment map $\mu: X \to P_X$). For any $v \in \operatorname{Vert}(P_X)$, define $\lambda_v^L = {}^{\operatorname{t}} (\lambda_{F_{v, 1}}^L, ..., \lambda_{F_{v, n}}^L)$, where $F_{v, 1}, ..., F_{v, n} \in \Sigma^{(1)}$ are the $n$ facets of $P_X$ adjacent to $v$ with the same ordering given as in Section \ref{Section 2}.\par
When $\lambda_F^L \in \mathbb{Z}$ for all $F \in \Sigma^{(1)}$, there furthermore associate a $T$-equivariant Hermitian line bundle $(L, h^L)$ with the following defining conditions. We have a local unitary frame $\mathbf{1}_{v, L}^{\operatorname{U}(1)}$ of $(L, h^L)$ over $U_v$ for each $v \in \overline{\operatorname{Vert}}(P_X)$ such that
\begin{itemize}
	\item if $v \in \operatorname{Vert}(P_X)$, then $\mathbf{1}_{v, L}^{\operatorname{U}(1)} \vert_{U_\circ \cap U_v} = e^{-\sqrt{-1} \lambda_v^L \cdot \theta_v} \mathbf{1}_{\circ, L}^{\operatorname{U}(1)} \vert_{U_\circ \cap U_v}$; and
	\item if $v, v' \in \operatorname{Vert}(P_X)$, then $\mathbf{1}_{v, L}^{\operatorname{U}(1)} \vert_{U_v \cap U_{v'}} = e^{\sqrt{-1} (\lambda_{v'}^L - A_{v'} A_v^{-1} \lambda_v^L) \cdot \theta_{v'}} \mathbf{1}_{v', L}^{\operatorname{U}(1)} \vert_{U_v \cap U_{v'}}$.
\end{itemize}
In particular, we can define a map
\begin{equation}
\Sigma^{(1)} \to H^2(X, \mathbb{Z}), \quad F \mapsto D_F := c_1(L),
\end{equation}
where $(L, h^L)$ is the $T$-equivariant Hermitian line bundle associated to the map $\Sigma^{(1)} \to \mathbb{Z}$ given by $F' \mapsto \lambda_{F'}^L$, where $\lambda_{F'}^L = 1$ if $F' = F$ and $\lambda_{F'}^L = 0$ otherwise. Then $\{D_F\}_{F \in \Sigma^{(1)}}$ forms a set of generators of $H^2(X, \mathbb{Z})$. We have $[\tfrac{\omega}{2\pi}] = \sum_{F \in \Sigma^{(1)}} \lambda_F D_F$ and the first Chern class of the symplectic manifold $(X, \omega)$ is given by
\begin{equation}
\label{Equation 2.3}
c_1(X) = \sum_{F \in \Sigma^{(1)}} D_F.
\end{equation}
Note that for a general map $\Sigma^{(1)} \to \mathbb{Z}, F \mapsto \lambda_F^L$, the associated line bundle $L$ is ample if and only if the normal fan of $P_L$ is the same as that of $P_X$, i.e. $\Sigma$.
% Note that the number of vertices of $P_L$ might not be the same as that of $P_X$.

\subsection{Fibrewise Fourier transform}
\quad\par
In this subsection, we shall introduce fibrewise Fourier transform of smooth functions on $X$, on which calculations in the later subsections heavily rely. It is essentially the same as that appeared in Theorem 1.3 in \cite{C2007} by Cadet, but we shall use a slightly different description.\par
The idea is that we know how to perform (fibrewise) Fourier transform of functions on $P_X \times T$ into functions on $P_X \times \mathbb{Z}^n$. We shall construct a continuous map $\operatorname{Bl}_X: P_X \times T \to X$ so that the diagram commutes:
\begin{center}
	\begin{tikzcd}
		P_X \times T \ar[rr, "\operatorname{Bl}_X"] \ar[rd] && X \ar[ld, "\mu"]\\
		& P_X
	\end{tikzcd}
\end{center}
where $P_X \times T \to P_X$ is the canonical projection, and pull back functions on $X$ along $\operatorname{Bl}_X$. This requires the chosen set of action-angle coordinates $\{(x_v, \theta_v)\}_{v \in \overline{\operatorname{Vert}(P_X)}}$ described in Section \ref{Section 2}. Indeed, $(x, \theta)$ are coordinates on $\mathring{X} \cong \mathring{P}_X \times T =: \tilde{U}_\circ$. Define a map
\begin{align*}
	\operatorname{Bl}: [0, +\infty)^n \times (\mathbb{R}^n/(2\pi \mathbb{Z})^n) & \to \mathbb{R}^{2n},\\
	(y^1, ..., y^n, \eta^1, ..., \eta^n) & \mapsto (\sqrt{y^1}\cos \eta^1, ..., \sqrt{y^n}\cos \eta^n, \sqrt{y^1}\sin \eta^1, ..., \sqrt{y^n}\cos \eta^n).
\end{align*}
For any vertex $v \in \operatorname{Vert}(P_X)$, we identify $U_v$ as a subset of $\mathbb{R}^{2n}$ via the embedding $(a_v, b_v)$ and take $\tilde{U}_v$ to be the preimage $\operatorname{Bl}^{-1}(U_v)$. On $\tilde{U}_v$, $(x_v, \theta_v)$ becomes a non-singular coordinate chart. In coordinates, the map $\operatorname{Bl} \vert_{\tilde{U}_v}$ is just given by $(x_v, \theta_v) \mapsto (a_v, b_v)$. Then $\{ (\tilde{U}_v, x_v, \theta_v) \}_{v \in \overline{\operatorname{Vert}(P_X)}}$ forms an atlas on the smooth manifold with corners $P_X \times T$. We glue the maps $\operatorname{Bl} \vert_{\tilde{U}_v}$ to a map $\operatorname{Bl}_X: P_X \times T \to X$, i.e. for all $v \in \operatorname{Vert}(P_X)$, the diagram
% The coordinates $(x, \theta)$ depends on a subset of $X$
\begin{center}
	\begin{tikzcd}
		\tilde{U}_v \ar[d] \ar[r, "\operatorname{Bl}_X"] & U_v \arrow[d, "(a_v{,} b_v)"]\\
		{[}0, +\infty{)}^n \times T \ar[r, "\operatorname{Bl}"'] & \mathbb{R}^{2n}
	\end{tikzcd}
\end{center}
commutes, where $\tilde{U}_v \hookrightarrow [0, +\infty)^n \times T$ is the inclusion map. $P_X \times T$ might be considered as a real blow-up %(\textcolor{red}{or real resolution? what name should be meaningful?})
of $X$ along $X \backslash \mathring{X}$ with a blow-down map $\operatorname{Bl}_X: P_X \times T \to X$.\par
At a point in the interior $\mathring{P}_X \times T$, $\operatorname{Bl}_X$ is clearly smooth. Observe, however, that on the boundary of $\tilde{U}_v$ with $v \in \operatorname{Vert}(P_X)$, the map $\operatorname{Bl} \vert_{\tilde{U}_v}: \tilde{U}_v \to U_v$ is only smooth in the last $n$ variables $\theta_v^1, ..., \theta_v^n$. Therefore, at a point in the boundary $\partial P_X \times T$, the blow-down map $\operatorname{Bl}_X$ is continuous but not differentiable. A consequence is that a smooth function $f$ on $X$ is pulled back via $\operatorname{Bl}_X$ to a continuous function on $P_X \times T$ which might not be smooth, although it is smooth on $\mathring{P}_X \times T \cong \mathring{X}$. Yet, the map $\operatorname{Bl}_X: P_X \times T \to X$ is still useful in defining fibrewise Fourier transform. 

\begin{definition}
	Let $f \in \mathcal{C}^\infty(X, \mathbb{C})$. For $p \in \mathbb{Z}^n$, the $p$\emph{th fibrewise Fourier coefficient} of $f$, denoted by $\widehat{f}_p$, is defined as the pushforward of $\frac{1}{(2\pi)^n} (\operatorname{Bl}_X^*f)(x, \theta) e^{-\sqrt{-1} p \cdot \theta} d^n\theta$ along $\mu \circ \operatorname{Bl}_X$:
	\begin{equation}
	\widehat{f}_p(x) := \frac{1}{(2\pi)^n} \int_T (\operatorname{Bl}_X^*f)(x, \theta) e^{-\sqrt{-1} p \cdot \theta} d^n\theta, \quad x \in P_X.
	\end{equation}
	The \emph{fibrewise Fourier transform} of $f$, denoted by $\widehat{f}$, is the defined as the function
	\begin{align*}
	P_X \times \mathbb{Z}^n \to \mathbb{C}, \quad (x, p) \mapsto \widehat{f}_p(x).
	\end{align*}
\end{definition}

\begin{remark}
	In Cadet's construction \cite{C2007} of Fourier transform of functions on $X$, one requires a continuous section of the fibration $\mu: X \to P_X$. Indeed, the set of action-angle coordinates that we fixed provides such a continuous section determined by the map $\mathring{P}_X \ni x \mapsto (x, 0) \in \mathring{X}$.
\end{remark}

% Restricted on $\mathring{X}$, $\mu$ is a smooth fibration $\mathring{X} \to \mathring{P}_X$. Since fibres of this restricted map are (diffeomorphic to) $T$, we have integration along fibre $\mu_*: \Omega^*(\mathring{X}, \mathbb{C}) \to \Omega^{*-n}(\mathring{P}_X, \mathbb{C})$: for all $\alpha \in \Omega^*(\mathring{X}, \mathbb{C})$ and $\beta \in \Omega^*(\mathring{P}_X, \mathbb{C})$, $\int_{\mathring{P}_X} \mu_*\alpha \wedge \beta = \int_{\mathring{X}} \alpha \wedge \mu^*\beta$. For $f \in \mathcal{C}^\infty(X, \mathbb{C})$ and $p \in \mathbb{Z}^n$, we define its $p$th fibrewise Fourier coefficient as the pushforward of $\frac{1}{(2\pi)^n} f(x, \theta) e^{-\sqrt{-1} p \cdot \theta} d^n\theta$ along $\mu$:
% \begin{equation}
% \widehat{f}_p(x) := \frac{1}{(2\pi)^n} \int_T f(x, \theta) e^{-\sqrt{-1} p \cdot \theta} d^n\theta, \quad x \in \mathring{P}_X,
% \end{equation}
% which is a smooth function on $\mathring{P}_X$.
We then see that for $f \in \mathcal{C}^\infty(X, \mathbb{C})$, $\widehat{f}$ satisfies the following properties. For any $p \in \mathbb{Z}^n$, $\widehat{f}_p$ is smooth on $\mathring{P}_X$ and continuous on $P_X$, but not smooth on $P_X$ in general. In addition, for $x \in P_X$, let $T_x < T$ be the common isotropy subgroup of all points in $\mu^{-1}(\{x\})$, i.e.
\begin{align*}
	T_x = \{ \eta \in T: \eta \cdot y = y \text{ for all } y \in \mu^{-1}(\{x\}) \},
\end{align*}
and $Z_x < \mathbb{Z}^n$ be the Pontryagin dual of $T/T_x$, i.e.
\begin{align*}
	Z_x = \{ p \in \mathbb{Z}^n: e^{\sqrt{-1} p \cdot \theta} = 1 \text{ for all } \theta \in T_x \}.
\end{align*}
Then for all $p \in \mathbb{Z}^n \backslash Z_x$, $\widehat{f}_p(x) = 0$. Therefore, we can regard $\widehat{f}$ as a function on $\bigsqcup_{x \in P_X} \{x\} \times Z_x$.\par
% If $p \in \mathbb{Z}^n \backslash Z_x$, then there exists $\theta_0 \in T_x$ such that $e^{\sqrt{-1}p \cdot \theta_0} \neq 1$. However, $\operatorname{Bl}_X^*f(x, \quad)$ is invariant under the subgroup $T_{\theta_0} = \{ c\theta_0 \in T_x: c \in \mathbb{R} \}$. It forces $\int_T (\operatorname{Bl}_X^*f)(x, \theta) e^{-\sqrt{-1} p \cdot \theta} d^n\theta = 0$.
Using fibrewise Fourier transform, we obtain a sequence $\{ \mathcal{C}^\infty(X, \mathbb{C})_{\leq l} \}_{l=0}^\infty$ of subspaces of $\mathcal{C}^\infty(X, \mathbb{C})$, where $\mathcal{C}^\infty(X, \mathbb{C})_{\leq l}$ is the set of functions $f \in \mathcal{C}^\infty(X, \mathbb{C})$ such that $\widehat{f}_p$ vanishes everywhere on $P_X$ for all $p = (p_1, ..., p_n) \in \mathbb{Z}^n$ with $\max_{1 \leq i \leq n} \lvert p_i \rvert > l$. Observe that for $i, j \in \mathbb{N} \cup \{0\}$,
\begin{align*}
	\mathcal{C}^\infty(X, \mathbb{C})_{\leq i} \cdot \mathcal{C}^\infty(X, \mathbb{C})_{\leq j} & \subset \mathcal{C}^\infty(X, \mathbb{C})_{\leq (i+j)},\\
	\{ \mathcal{C}^\infty(X, \mathbb{C})_{\leq i}, \mathcal{C}^\infty(X, \mathbb{C})_{\leq j} \} & \subset \mathcal{C}^\infty(X, \mathbb{C})_{\leq (i+j)}.
\end{align*}
Therefore, $\mathcal{C}^\infty(X, \mathbb{C})_{<\infty} := \bigcup_{l=0}^\infty \mathcal{C}^\infty(X, \mathbb{C})_{\leq l}$ is a Poisson subalgebra of $(\mathcal{C}^\infty(X, \mathbb{C}), \{\quad, \quad\})$ and is itself a filtered Poisson algebra with the ascending filtration
\begin{align*}
\mathcal{C}^\infty(X, \mathbb{C})_{\leq 0} \subset \cdots \subset \mathcal{C}^\infty(X, \mathbb{C})_{\leq l} \subset \cdots \subset \mathcal{C}^\infty(X, \mathbb{C})_{<\infty}.
% \mathcal{C}^\infty(P_X, \mathbb{C}) \cong \mathcal{C}^\infty(X, \mathbb{C})_{\leq 0}?
\end{align*}
% the completion of $\mathcal{C}^\infty(X, \mathbb{C})_{<\infty}$ is $\mathcal{C}^\infty(X, \mathbb{C})$?
We end this section by stating estimates of norms of fibrewise Fourier cooefficients of functions on $X$ and their derivatives, which will be used in the proof of our main theorem in Section \ref{Section 5}. For $m \in \mathbb{Z}^n$, define
\begin{equation}
\label{Equation 5.2}
[m] = \prod_{m_i \neq 0} (2\pi \lvert m_i \rvert).
\end{equation}

\begin{lemma}
	\label{Lemma 5.3}
	Let $f \in \mathcal{C}^\infty(X, \mathbb{C})$, $I$ be any multi-index and $r \in \mathbb{N}$. The following statements hold.
	\begin{enumerate}
		\item There exists $C_{f, I, r} > 0$ such that for all $m \in \mathbb{Z}^n$ and $x \in \mathring{P}_X$, $\left\lvert m^I \widehat{f}_m(x) \right\vert \leq C_{f, I, r} [m]^{-r}$.
		\item If $K$ is a compact subset of $\mathring{P}_X$, then there exists $C_{f, I, r, K} > 0$ such that for all $m \in \mathbb{Z}^n$ and $x \in K$, $\left\lvert \widehat{f}_m^{(I)}(x) \right\vert \leq C_{f, I, r, K} [m]^{-r}$, where $\widehat{f}_m^{(I)}(x) = \tfrac{\partial^{\lvert I \rvert}}{\partial x^I} \widehat{f}_m(x)$.
	\end{enumerate}
\end{lemma}
\begin{proof}
	The proof is the same as that of Lemma 5.3 in \cite{LY2021} for symplectic tori vertibum (see also \cite{NS1991, SS2002}). We only need to note that $\tfrac{\partial^{\lvert I \rvert} f}{\partial \theta^I}$ is a globally defined smooth function on the compact manifold $X$ whose $m$th fibrewise Fourier coefficient of $\tfrac{\partial^{\lvert I \rvert} f}{\partial \theta^I}$ is $(\sqrt{-1})^{\lvert I \rvert} m^I \hat{f}_m$ for each $m \in \mathbb{Z}^n$, whereas $\tfrac{\partial^{\lvert I \rvert} f}{\partial x^I}$ is a smooth function on the compact subset $\mu^{-1}(K)$ of $\mathring{X}$ whose $m$th fibrewise Fourier coefficient of $\widehat{\frac{\partial^{\lvert I \rvert} f}{\partial x^I}}$ is $\widehat{f}_m^{(I)}$ for each $m \in \mathbb{Z}^n$.
\end{proof}

\subsection{Real polarizations}
\quad\par
In this subsection, we shall introduce two singular real polarizations $\mathcal{P}^{\operatorname{v}}$ and $\mathcal{P}^{\operatorname{h}}$ that will be used in the construction of Toeplitz-type operators in Section \ref{Section 5}.\par
The $T$-action on $(X, \omega)$ induces a morphism of Lie algebroids
\begin{align*}
\chi: X \times \mathfrak{t} \to TX, \quad (p, \xi) \mapsto \chi_\xi(p),
\end{align*}
where $\chi_\xi$ is the fundamental vector field of $\xi$, from the Lie algebra bundle $X \times \mathfrak{t}$ to the tangent bundle $TX$. We then have a sequence of vector bundle maps
\begin{center}
	\begin{tikzcd}
		X \times \mathfrak{t} \ar[r, "\chi"] & TX \ar[r, "d\mu"] & \mu^*TP_X
	\end{tikzcd}
\end{center}
and for any point $p \in X$, $\chi(\{p\} \times \mathfrak{t})$ and $\ker d\mu(p)$ are isotropic and coisotropic subspaces of $T_pX_\mathbb{C}$ respectively such that $\chi(\{p\} \times \mathfrak{t}) \subset \ker d\mu(p)$. 
% $(d\mu) (\partial_{a_v^i}) = 2a_v^i \partial_{x_v^i}$ and $(d\mu) (\partial_{b_v^i}) = 2b_v^i \partial_{x_v^i}$.
% On a face $F$ of $P_X$, $d\mu: TX \vert_{\mathring{F}} \to T\mathring{F}$ is surjective.
When restricted on $\mathring{X}$, the above sequence is a short exact sequence of vector bundles over $\mathring{X}$.\par
There are two conventions on defining the real toric polarization $\mathcal{P}^{\operatorname{v}}$. In \cite{BFMN2011}, $\mathcal{P}^{\operatorname{v}}$ is defined to be the kernel $(\ker d\mu)_\mathbb{C}$; in \cite{KMN2013}, $\mathcal{P}^{\operatorname{v}}$ is defined to be the image of $X \times \mathfrak{t}_\mathbb{C}$ under the complexification $\chi_\mathbb{C}$ of $\chi$. Indeed, this is a minor issue. The choice of conventions makes no essential difference for the construction of Toeplitz-type operators in Section \ref{Section 5} - what is more important is the space of smooth sections of $\mathcal{P}^{\operatorname{v}}$. 
% (In \cite{BFMN2011}, $\mathcal{P}_\mathbb{R}$ stands for the vertical polarization $(\ker d\mu)_\mathbb{C}$; in \cite{KMN2013}, $\mathcal{P}_\mathbb{R}$ stands for the image of the complexification of $\chi$).
We adopt the latter convention that $\mathcal{P}^{\operatorname{v}}$ is given by the image of $X \times \mathfrak{t}_\mathbb{C}$ under $\chi_\mathbb{C}$, i.e.
\begin{align*}
\mathcal{P}_p^{\operatorname{v}} = \operatorname{span}_\mathbb{C} \left\{ \left. \partial_{\theta^1} \right\vert_p, ..., \left. \partial_{\theta^n} \right\vert_p \right\}, \quad \text{for all } p \in X.
\end{align*}
By a smooth section of $\mathcal{P}^{\operatorname{v}}$ over an open subset $U$ of $X$, we mean a smooth section $\xi$ of $TX_\mathbb{C}$ over $U$ such that for all $p \in U$, $\xi_p \in \mathcal{P}_p^{\operatorname{v}}$. We denote by $\Gamma(U, \mathcal{P}^{\operatorname{v}})$ the space of smooth sections of $\mathcal{P}^{\operatorname{v}}$ over $U$. Then $\Gamma(U, \mathcal{P}^{\operatorname{v}})$ is closed under Lie bracket $[\quad, \quad]$.\par
While the real polarization $\mathcal{P}^{\operatorname{v}}$ is evidently independent of the choice of action-angle coordinates, the chosen action-angle coordinates $(x, \theta)$ induces a real polarization $\mathcal{P}^{\operatorname{h}}$ on $(\mathring{X}, \omega \vert_{\mathring{X}})$ given by
\begin{align*}
\mathcal{P}_p^{\operatorname{h}} = \operatorname{span}_\mathbb{C} \left\{ \left. \partial_{x^1} \right\vert_p, ..., \left. \partial_{x^n} \right\vert_p \right\}, \quad \text{for all } p \in \mathring{X}.
\end{align*}
Note that $\mathcal{P}^{\operatorname{h}}$ is transversal to $\mathcal{P}^{\operatorname{v}} \vert_{\mathring{X}}$. As $\mathcal{P}^{\operatorname{h}}$ is only defined on the open dense subset $\mathring{X}$ of $X$, we consider it as a singular polarization on $X$.

\section{Deformation Quantization Compatible with Real Polarizations}
\label{Section 3}
In this section, we shall define star products on $(X, \omega)$ that appear in our main theorem.\par %Let $\hbar$ be a formal variable.\par
From the observation in the case of symplectic tori \cite{LY2021}, for the purpose of compatibility with $\mathcal{P}^{\operatorname{v}}$ and $\mathcal{P}^{\operatorname{h}}$, we might expect a star product $\star$ on $(X, \omega)$ of the form
\begin{equation}
\label{Equation 2.4}
f \star g = \sum_{r=0}^\infty \hbar^r C_r(f, g),
\end{equation}
where for any $r \in \mathbb{N} \cup \{0\}$,
\begin{equation}
\label{Equation 2.5}
C_r(f, g) = \frac{1}{r! \cdot (\sqrt{-1})^r} \sum_{i_1=1}^n \cdots \sum_{i_r=1}^n \frac{\partial^r f}{\partial x^{i_1} \cdots \partial x^{i_r}} \frac{\partial^r g}{\partial \theta^{i_1} \cdots \partial \theta^{i_r}} = \frac{1}{(\sqrt{-1})^r} \sum_{\lvert I \rvert = r} \frac{1}{I!} \frac{\partial^r f}{\partial x^I} \frac{\partial^r g}{\partial \theta^I}.
\end{equation}
with respect to the action-angle coordinates $(x, \theta)$ on $\mathring{X}$. This is well-defined on $(\mathring{X}, \omega \vert_{\mathring{X}})$ and for any vertex $v \in \operatorname{Vert}(P_X)$, as we see from (\ref{Equation 3.1}) that as $A_v, \lambda_v$ are constants, we can rewrite $C_r(f, g)$ ($r \in \mathbb{N} \cup \{0\}$) in terms of the vertex coordinate chart $(x_v, \theta_v)$:
\begin{align*}
C_r(f, g) = \frac{1}{(\sqrt{-1})^r} \sum_{\lvert I \rvert = r} \frac{1}{I!} \frac{\partial^r f}{\partial x_v^I} \frac{\partial^r g}{\partial \theta_v^I}.
\end{align*}
However, the limit of $C_r(f, g)$ might not exist as the point approaches $X \backslash \mathring{X}$ because of singularity of $\tfrac{\partial^r f}{\partial x_v^I}$. Note that $\tfrac{\partial^r g}{\partial \theta_v^I}$ is well defined even on $X \backslash \mathring{X}$.

For the purpose of defining a Toeplitz-type operator in real toric polarization, it turns out that it is enough to only define the star product on $\mathring{X}$. This will be explained in Section \ref{Section 5}. Now we give the following definition.

\begin{definition}
	Let $\hbar$ be a formal variable. We define a $\mathbb{C}[[\hbar]]$-bilinear map
	\begin{align*}
		\star: \mathcal{C}^\infty(X, \mathbb{C})[[\hbar]] \times \mathcal{C}^\infty(X, \mathbb{C})[[\hbar]] \to \mathcal{C}^\infty(\mathring{X}, \mathbb{C})[[\hbar]], \quad (f, g) \mapsto f \star g
	\end{align*}
	as follows: for all $f, g \in \mathcal{C}^\infty(X, \mathbb{C})$,
	\begin{equation*}
	f \star g = \sum_{r=0}^\infty \hbar^r C_r(f, g),
	\end{equation*}
	where $C_r(f, g)$'s are defined as in (\ref{Equation 2.5}). We call $\star$ the \emph{(singular) toric star product}.
\end{definition}
Similar to $\mathcal{C}^\infty(X, \mathbb{C})$, we can define $\mathcal{C}^\infty(\mathring{X}, \mathbb{C})_{\leq l}$ for each $l \in \mathbb{N} \cup \{0\}$ and $\mathcal{C}^\infty(\mathring{X}, \mathbb{C})_{<\infty}$ via fibrewise Fourier transform. The toric star product $\star$ has the property that for $i, j \in \mathbb{N} \cup \{0\}$,
\begin{align*}
	\mathcal{C}^\infty(X, \mathbb{C})_{\leq i}[[\hbar]] \star \mathcal{C}^\infty(X, \mathbb{C})_{\leq j}[[\hbar]] \subset \mathcal{C}^\infty(\mathring{X}, \mathbb{C})_{\leq (i+j)}[[\hbar]].
\end{align*}
%Therefore, the star product $\star$ restricts to a map
%\begin{align*}
%	\mathcal{C}^\infty(X, \mathbb{C})_{<\infty}[[\hbar]] \times \mathcal{C}^\infty(X, \mathbb{C})_{<\infty}[[\hbar]] \to \mathcal{C}^\infty(\mathring{X}, \mathbb{C})_{<\infty}[[\hbar]].
%\end{align*}
% Therefore, $\mathcal{C}^\infty(X, \mathbb{C})_{< \infty}[[\hbar]]$ is a `subalgebra' of $(\mathcal{C}^\infty(X, \mathbb{C})[[\hbar]], \star)$ and is itself a `filtered algebra'.
% Our star product is not a short star product.
% http://www.unige.ch/math/folks/nikolaev/assets/files/GP-20200423-PavelEtingof.pdf

\section{Geometric Quantization}
\label{Section 4}
In this section, we shall perform geometric quantization on $(X, \omega)$ in real polarization $\mathcal{P}^{\operatorname{v}}$. In order to establish relations between deformation quantization and geometric quantization in later discussion, we need to consider scaling of the symplectic form $\omega$ and study the behavior of quantization near the large volume limit. Now, instead of regarding $\hbar$ as a formal variable, we let $k \in \mathbb{N}$ and $\hbar = \tfrac{1}{k}$.\par
We know from \cite{BFMN2011} that the usual way of performing geometric quantization directly on a prequantum line bundle on $(X, \omega)$ requires the choice of $\lambda_F$ to be integers and there yields Bohr-Sommerfeld fibres over boundary lattice points of the moment polytope $P_X$. A drawback of this is when we mimic the construction of Toeplitz-type operators in real polarizations as in \cite{LY2021} and study norm estimations, we need to handle ill-defined values of the terms $C_r(f, g)$ in (\ref{Equation 2.5}) at points in $X \backslash \mathring{X}$.\par
To avoid this technical problem, we instead adopt \emph{half-form correction} (also known as \emph{metaplectic correction}), as established in \cite{KMN2013}. We shall first give a brief review on the general procedures for geometric quantization with half-form correction, and then state clearly the assumptions to be made for our quantization scheme. For illustration, suppose $[\tfrac{\omega}{2\pi}] = c_1(L) \in H^2(X, \mathbb{Z})$ with $(L, h^L, \nabla)$ a prequantum line bundle, $\mathcal{P}$ is a non-singular polarization on $(X, \omega)$, and $\tfrac{1}{2} c_1(X) \in H^2(X, \mathbb{Z})$, or equivalently, $(X, \omega)$ admits a metaplectic structure. Recall that the isomorphism classes of metaplectic structures on $(X, \omega)$ are classified by $H^1(X, \mathbb{Z}_2)$, which is trivial for toric varieties. Thus, we can pick the unique metaplectic structure.\par
As $\mathcal{P}$ is a complex Lagrangian distribution on $(X, \omega)$, it corresponds to a complex line bundle $K^\mathcal{P}$ defined by
\begin{equation}
\label{Equation 4.1}
K_p^\mathcal{P} = \left\{ \alpha \in \textstyle\bigwedge^n T_p^*X_\mathbb{C}: \iota_{\overline{\xi}} \alpha = 0 \text{ for any } \xi \in \mathcal{P}_p \right\},
\end{equation}
for any $p \in X$. We call $K^\mathcal{P}$ the \emph{canonical line bundle} of $\mathcal{P}$. It has a canonical flat $\overline{\mathcal{P}}$-connection $\nabla^{K^\mathcal{P}}$ induced by the Bott connection on the quotient vector bundle $TX_\mathbb{C}/\overline{\mathcal{P}}$. The metaplectic structure induces a square root $(\sqrt{K^\mathcal{P}}, \nabla^{\sqrt{K^\mathcal{P}}})$ of $(K^\mathcal{P}, \nabla^{K^\mathcal{P}})$. The half-form corrected prequantum line bundle (in level $k$) is the complex line bundle
\begin{align*}
L^{\otimes k} \otimes \sqrt{K^\mathcal{P}}
\end{align*}
together with the $\overline{\mathcal{P}}$-connection $\nabla^{L^{\otimes k} \otimes \sqrt{K^\mathcal{P}}}$ induced by $\nabla$ and $\nabla^{\sqrt{K^\mathcal{P}}}$.\par
As explained in Section 2.1 in \cite{KMN2013}, every complex line bundle with connection $(E, \nabla^E)$ is canonically isomorphic to the tensor product of its \emph{$\operatorname{U}(1)$-part} $(E^{\operatorname{U}(1)}, \nabla^{E^{\operatorname{U}}})$, which has $\operatorname{U}(1)$-transition functions and curvature $2$-form lying in $\Omega^1(X, \sqrt{-1}\mathbb{R})$, and its \emph{modulus part} $(\lvert E \rvert, \nabla^{\lvert E \rvert})$, which has $\mathbb{R}^+$-transition functions and curvature $2$-form lying in $\Omega^1(X, \mathbb{R})$. Similar, we decompose
\begin{align*}
(\sqrt{K^\mathcal{P}}, \nabla^{\sqrt{K^\mathcal{P}}}) \cong ((\sqrt{K^\mathcal{P}})^{\operatorname{U}(1)}, \nabla^{(\sqrt{K^\mathcal{P}})^{\operatorname{U}(1)}}) \otimes (\sqrt{\lvert K^\mathcal{P} \rvert}, \nabla^{\sqrt{\lvert K^\mathcal{P} \rvert}})
\end{align*}
into the $\operatorname{U}(1)$-part $((\sqrt{K^\mathcal{P}})^{\operatorname{U}(1)}, \nabla^{(\sqrt{K^\mathcal{P}})^{\operatorname{U}(1)}})$ and its modulus part $(\sqrt{\lvert K^\mathcal{P} \rvert}, \nabla^{\sqrt{\lvert K^\mathcal{P} \rvert}})$, with the notion `connection' replaced by `$\overline{\mathcal{P}}$-connection'. 
% $L^{\otimes k} \otimes \sqrt{K^\mathcal{P}} \cong L^{\otimes k} \otimes (\sqrt{K^\mathcal{P}})^{\operatorname{U}(1)} \otimes \sqrt{\lvert K^\mathcal{P} \rvert}$
Note that $c_1(\sqrt{\lvert K^\mathcal{P} \rvert}) = 0$ and it is computed in Section 3.1 in \cite{KMN2013} that when $\mathcal{P}$ is the K\"ahler polarization induced by a $T$-invariant complex structure on $X$,
\begin{align*}
c_1(L^{\otimes k} \otimes (\sqrt{K^\mathcal{P}})^{\operatorname{U}(1)}) = k[\tfrac{\omega}{2\pi}] - \tfrac{1}{2} c_1(X).
\end{align*}
\par
Indeed, there is no need to restrict to the case when the first Chern class $c_1(X)$ of $(X, \omega)$ is even. Here and in the squeal, we only assume $[\tfrac{\omega}{2\pi}] - \tfrac{1}{2} c_1(X) \in H^2(X, \mathbb{Z})$. Then $k[\tfrac{\omega}{2\pi}] \in H^2(X, \mathbb{Z})$ for $k$ even, therefore $k[\tfrac{\omega}{2\pi}] - \tfrac{1}{2} c_1(X) \in H^2(X, \mathbb{Z})$ for $k$ odd. Also note that there exists $k_0 \in \mathbb{N}$ such that for all $k \in \mathbb{N}$ with $k \geq k_0$, $k[\tfrac{\omega}{2\pi}] - \tfrac{1}{2} c_1(X)$ is a positive class (with respect to a $T$-invariant complex structure determined by a symplectic potential). Thus, we furthermore assume $k$ is odd and $k \geq k_0$. In the coming subsection, we directly construct a Hermitian line bundle $(L_+^k, h^{L_+^k})$ with $\overline{\mathcal{P}}$-connection playing the role of $L^{\otimes k} \otimes (\sqrt{K^\mathcal{P}})^{\operatorname{U}(1)}$ without defining $L$ and $(\sqrt{K^\mathcal{P}})^{\operatorname{U}(1)}$ separately and explain the meaning of $\sqrt{\lvert K^\mathcal{P} \rvert}$ when $\mathcal{P} = \mathcal{P}^{\operatorname{v}}$ is a singular polarization.

\subsection{The half-form corrected prequantum line bundle}
\quad\par
We first explain the meaning of $\sqrt{\lvert K^{\mathcal{P}^{\operatorname{v}}} \rvert}$. Since $\mathcal{P}^{\operatorname{v}}$ is singular on $X \backslash \mathring{X}$, we cannot have an honest complex line bundle $K^{\mathcal{P}^{\operatorname{v}}}$ serving as the canonical line bundle of $\mathcal{P}^{\operatorname{v}}$. However, we have a globally defined $n$-form $d^nx = dx^1 \wedge \cdots \wedge dx^n$ on $X$ with zero locus $X \backslash \mathring{X}$, 
% This is because $dx_v^i = a_v^i da_v^i + b_v^i db_v^i$.
which annihilates $\overline{\mathcal{P}^{\operatorname{v}}} = \mathcal{P}^{\operatorname{v}}$. If $p \in X$ and $\eta \in \bigwedge^n T_p^*X_\mathbb{C}$, we define $\sqrt{\lvert \eta \rvert}$ by the map
\begin{align*}
\textstyle\bigwedge^n T_pX  \to \mathbb{R}, \quad (\xi_1, ..., \xi_n) \mapsto \sqrt{\lvert \eta(\xi_1, ..., \xi_n) \rvert}.
\end{align*}
In particular, we have a map $\sqrt{\lvert d^nx \rvert}: \mathcal{X}(X)^n \to \mathcal{C}^\infty(X)$. Although the complex line bundle $\sqrt{\lvert K^{\mathcal{P}^{\operatorname{v}}} \rvert}$ is ill-defined, it suffices to make sense of its global smooth sections, and we regard them as elements in $\mathcal{C}^\infty(X, \mathbb{C}) \sqrt{\lvert d^n x\rvert}$, where we suppress $\otimes$ in $\mathcal{C}^\infty(X, \mathbb{C}) \otimes \sqrt{\lvert d^nx \rvert}$ and we shall keep on doing so for similar notations in the sequel.\par
%maps of the form
%\begin{equation*}
%f\sqrt{\lvert d^n x \rvert}: \mathcal{X}(X)^n \to \mathcal{C}^\infty(X, \mathbb{C}), \quad (\xi_1, ..., \xi_n) \mapsto f\sqrt{\lvert (d^n x)(\xi_1, ..., \xi_n) \rvert},
%\end{equation*}
%with $f \in \mathcal{C}^\infty(X, \mathbb{C})$.
There is no need to define the $\operatorname{U}(1)$-part $(\sqrt{K^{\mathcal{P}^{\operatorname{v}}}})^{\operatorname{U}(1)}$ as it will be absorbed in the line bundle $L_+^k$ to be defined now. By the assumption that $[\tfrac{\omega}{2\pi}] - \tfrac{1}{2} c_1(X) \in H^2(X, \mathbb{Z})$, we fix the choice of $\lambda_F$'s so that $\lambda_F \in \mathbb{Z} + \tfrac{1}{2}$ for any $1$-cone $F \in \Sigma^{(1)}$. Define a map $\Sigma^{(1)} \to \mathbb{Z}$ by $F \mapsto \lambda_F^{L_+^k} := k\lambda_F - \tfrac{1}{2}$ (note that $\lambda_F^{L_+^k} \in \mathbb{Z}$ since $k$ is assumed to be odd). Then $k[\tfrac{\omega}{2\pi}] - \frac{1}{2} c_1(X) = \sum_{F \in \Sigma^{(1)}} \lambda_F^{L_+^k} D_F$ by (\ref{Equation 2.3}). We construct a $T$-equivariant Hermitian line bundle $(L_+^k, h^{L_+^k})$ associated to $\{ \lambda_F^{L_+^k} \}_{F \in \Sigma^{(1)}}$ following the procedure described in Subsection \ref{Subsection 2.2}. 
% As explained in Remark 4.2 in \cite{KMN2013},
While the $T$-equivariant line bundle $L_+^k$ corresponds to a Delzant polytope $P_{L_+^k}$,
% need ampleness to have Delzant polytope?
\begin{equation*}
P_{X, \hbar} := kP_X = \{ x \in \mathfrak{t}^*: \text{for all } F \in \Sigma^{(1)}, l_{F, \hbar}(x) := \langle x, \nu_F \rangle + k\lambda_F \geq 0 \}
\end{equation*}
is the moment polytope for the moment map $\mu_\hbar := k\mu: X \to \mathfrak{t}^*$ with respect to the symplectic form $k\omega$. We see that $P_{L_+^k}$ is contained in the interior of $P_{X, \hbar}$ and the set $\Lambda_\hbar = P_{L_+^k} \cap \mathbb{Z}^n$ of lattice points in $P_{L_+^k}$ coincides with the set of lattice points in $P_{X, \hbar}$.\par
As a result, the ($k$-level) half-form corrected prequantum line bundle in polarization $\mathcal{P}^{\operatorname{v}}$ is $L_+^k \otimes \sqrt{\lvert K^{\mathcal{P}^{\operatorname{v}}} \rvert}$. Again, this is not an honestly defined line bundle, but it is enough for us to define its global smooth sections, which are regarded as elements in $\Gamma(X, L_+^k) \sqrt{\lvert d^n x \rvert}$.
% We denote it by $\Gamma(X, L_+^k \otimes \sqrt{K^{\mathcal{P}^{\operatorname{v}}}})$.
% maps of the form
% \begin{equation*}
% s \otimes \sqrt{\lvert d^n x \rvert}: \mathcal{X}(X)^n \to \Gamma(X, L_+^k), \quad (\xi_1, ..., \xi_n) \mapsto \sqrt{\lvert (d^n x)(\xi_1, ..., \xi_n) \rvert} s,
% \end{equation*}
% with $s \in \Gamma(X, L_+^k)$.

\subsection{Polarized sections and quantum Hilbert spaces}
\quad\par
To form the quantum Hilbert space, in principle we need to take $\mathcal{P}^{\operatorname{v}}$-polarized sections of $L_+^k \otimes \sqrt{\lvert K^{\mathcal{P}^{\operatorname{v}}} \rvert}$. Hence, we need to define flat $\overline{\mathcal{P}^{\operatorname{v}}}$-connections on $L_+^k$ and $\sqrt{\lvert K^{\mathcal{P}^{\operatorname{v}}} \rvert}$.\par
First, we have a canonical flat $\overline{\mathcal{P}^{\operatorname{v}}}$-connection $\nabla^{\sqrt{\lvert K^{\mathcal{P}^{\operatorname{v}}} \rvert}}$ on $\sqrt{\lvert K^{\mathcal{P}^{\operatorname{v}}} \rvert}$,
\begin{align*}
\Gamma(X, \overline{\mathcal{P}^{\operatorname{v}}}) \times (\mathcal{C}^\infty(X, \mathbb{C}) \sqrt{\lvert d^nx \rvert}) & \to \mathcal{C}^\infty(X, \mathbb{C}) \sqrt{\lvert d^nx \rvert},\\
(\xi, f \sqrt{\lvert d^n x\rvert}) & \mapsto \nabla_\xi^{\sqrt{\lvert K^{\mathcal{P}^{\operatorname{v}}} \rvert}} (f \sqrt{\lvert d^n x\rvert}) := (\mathcal{L}_\xi f) \sqrt{\lvert d^n x\rvert}.
\end{align*}
Also, $\sqrt{\lvert d^n x\rvert}$ is $\nabla^{\sqrt{\lvert K^{\mathcal{P}^{\operatorname{v}}} \rvert}}$-\emph{flat} in the sense that for any $\xi \in \Gamma(X, \mathcal{P}^{\operatorname{v}})$, $\nabla_{\overline{\xi}}^{\sqrt{\lvert K^{\mathcal{P}^{\operatorname{v}}} \rvert}} \sqrt{\lvert d^n x\rvert} = 0$.\par
Second, for any section $\tau \in \Gamma(X, L_+^k)$ and $v \in \overline{\operatorname{Vert}}(P_X)$, we define the function $\tau_v$ on $U_v$ by
\begin{equation}
	\label{Equation 4.2}
	\tau \vert_{U_v} = \tau_v \mathbf{1}_{v, L_+^k}^{\operatorname{U}(1)}.
\end{equation}
% \textcolor{blue}{$L_+^k \cong l^{\otimes \hbar^{-1}} \otimes \sqrt{K}^{\operatorname{U}(1)}$}.
Justified by \cite{KMN2013}, we introduce a flat $\overline{\mathcal{P}^{\operatorname{v}}}$-connection on $L_+^k$
\begin{align*}
	\nabla^\hbar: \Gamma(X, \overline{\mathcal{P}^{\operatorname{v}}}) \times \Gamma(X, L_+^k) \to \Gamma(X, L_+^k), \quad (\xi, s) \mapsto \nabla_\xi^\hbar s,
\end{align*}
which is determined by the following equalities: for $i \in \{1, ..., n\}$, $\tau \in \Gamma(X, L_+^k)$ and $v \in \operatorname{Vert}(P_X)$,
\begin{align}
\left. \nabla_{\partial_{\theta^i}}^\hbar \tau \right\vert_{\mathring{X}} = & \left( \frac{\partial \tau_\circ}{\partial \theta^i} - \sqrt{-1} k x^i \tau_\circ \right) \mathbf{1}_{\circ, L_+^k}^{\operatorname{U}(1)},\\
\left. \nabla_{\partial_{\theta_v^i}}^\hbar \tau \right\vert_{U_v} = & \left( \frac{\partial \tau_v}{\partial \theta_v^i} - \sqrt{-1} k x_v^i \tau_v + \frac{\sqrt{-1}}{2} \tau_v \right) \mathbf{1}_{v, L_+^k}^{\operatorname{U}(1)}.
\end{align}
% We need to identify different fibres of the line bundle over base points within the interior $\mathring{P}_X$, making use of the canonical section $\mathbf{1}_{\circ, L_+^k}^{\operatorname{U}(1)}$. For an arbitrary section $\tau \in \Gamma(X, L_+^k)$, we write $\tau = \tau_\circ \mathbf{1}_{\circ, L_+^k}^{\operatorname{U}(1)}$ on $\mathring{X}$.
As we deal with real polarization $\mathcal{P}^{\operatorname{v}}$, we need to consider distributional sections of $L_+^k$. For any open subset $U$ of $X$, $s \in \Gamma(U, \overline{L}_+^k)'$ and $\tau \in \Gamma(U, L_+^k)$, we define $\langle s, \tau \rangle = s(\overline{\tau})$. 
% and $\langle \tau, s \rangle = \overline{\langle s, \tau \rangle}$.
We also embed $\Gamma(U, L_+^k)$ into $\Gamma(U, \overline{L}_+^k)'$ via the Hermitian metric $h^{L_+^k}$ and the Liouville measure. Then the flat $\overline{\mathcal{P}^{\operatorname{v}}}$-connection $\nabla^\hbar$ extends to act on distributional sections of $L_+^k$. By a \emph{$\mathcal{P}^{\operatorname{v}}$-polarized (distributional) section} of $L_+^k$, we mean a distributional section $s \in \Gamma(X, \overline{L}_+^k)'$ such that $\nabla_{\overline{\xi}}^\hbar s = 0$ for any $\xi \in \Gamma(X, \mathcal{P}^{\operatorname{v}})$.\par
The above discussions motivate the following definition (Definition 4.3 in \cite{KMN2013}).
\begin{definition}
	The \emph{vector space of quantum states for the half-form corrected quantization in the real singular toric polarization} $\mathcal{P}^{\operatorname{v}}$ is defined as $\mathcal{H}^\hbar = \mathcal{B}^\hbar \otimes \sqrt{\lvert d^nx \rvert}$, where $\mathcal{B}^\hbar$ is the space of $\mathcal{P}^{\operatorname{v}}$-polarized distributional sections of $L_+^k$.
\end{definition}
By Theorem 4.5 in \cite{KMN2013},
\begin{equation}
	\mathcal{H}^\hbar = \bigoplus_{m \in \Lambda_\hbar} \mathbb{C} \delta_\hbar^m \sqrt{\lvert d^n x \rvert},
\end{equation}
where for all $m \in \Lambda_\hbar$, $\delta_\hbar^m \in \Gamma(X, \overline{L}_+^k)'$ is the distributional section of $L_+^k$ defined by
\begin{equation}
\langle \delta_\hbar^m, \tau \rangle = \frac{1}{(2\pi)^n} \int_T e^{\sqrt{-1} m \cdot \theta} \overline{\tau}_\circ(\hbar m, \theta) d^n\theta,
\end{equation}
where $\tau_\circ$ is defined in (\ref{Equation 4.2}), for any $\tau \in \Gamma(X, L_+^k)$.  %\footnote{to explain the notion $x_\hbar = m$. No need now as for $k$ odd, $x_\hbar = kx$}.
As explained in Remark 4.14 in \cite{KMN2013}, we define an inner product on $\mathcal{H}^\hbar$ by declaring
\begin{equation*}
\{ \sigma_\hbar^m := (4\pi)^{\frac{n}{4}} \delta_\hbar^m \sqrt{\lvert d^nx \rvert} \}_{m \in \Lambda_\hbar}
\end{equation*}
to be an orthonormal basis of $\mathcal{H}^\hbar$.

\section{Berezin-Toeplitz Quantization in Real Polarizations}
\label{Section 5}
This is the main section of this paper. We continue to assume that $[\tfrac{\omega}{2\pi}] - \tfrac{1}{2} c_1(X) \in H^2(X, \mathbb{Z})$, $k \in \mathbb{N}$ is odd and $k \geq k_0$. In Subsection \ref{Subsection 5.1}, after explanation of ingredients involved in the definition, we shall define a Toeplitz-type operator
\begin{align*}
Q^\hbar: \mathcal{C}^\infty(X, \mathbb{C}) \times \mathcal{H}^\hbar \to \mathcal{H}^\hbar, \quad (f, s) \mapsto Q_f^\hbar s
\end{align*}
acting on the quantum Hilbert space $\mathcal{H}^\hbar$ in the singular real polarization $\mathcal{P}^{\operatorname{v}}$. In Subsection \ref{Subsection 5.2}, we describe \emph{locality} of the Toeplitz-type operators as so to justify the meaning of the norm estimations on their localization. Finally in Subsection \ref{Subsection 5.3}, we provide proofs of our main theorem.

\subsection{Construction of Toeplitz operators for real polarizations}
\label{Subsection 5.1}
\quad\par
Observe that for $f \in \mathcal{C}^\infty(X, \mathbb{C})$ and $s \in \mathcal{H}^\hbar$, $fs \in \Gamma(\mathring{X}, \overline{L}_+^k)' \sqrt{\lvert d^n x \rvert}$. We are motivated by the case of symplectic tori \cite{LY2021} that in order to define Toeplitz-type operators, we need to construct a projection map
\begin{align*}
	\Pi^\hbar: \Gamma(\mathring{X}, \overline{L}_+^k)' \sqrt{\lvert d^n x \rvert} \to \mathcal{H}^\hbar,
\end{align*}
and the construction requires a vector space $\check{\mathcal{H}}^\hbar$ that is expected to be the quantum Hilbert space in the real polarization $\mathcal{P}^{\operatorname{h}}$ which is transversal to $\mathcal{P}^{\operatorname{v}}$, and a pairing between $\Gamma(\mathring{X}, \overline{L}_+^k)' \otimes \sqrt{\lvert d^n x \rvert}$ and $\check{\mathcal{H}}^\hbar$ generalizing the \emph{Blattner-Kostant-Sternberg (BKS) pairing}.\par
Now we perform geometric quantization on $\mathring{X}$ in real polarization $\mathcal{P}^{\operatorname{h}}$. The canonical line bundle $K^{\mathcal{P}^{\operatorname{h}}}$ of $\mathcal{P}^{\operatorname{h}}$ is well defined on $\mathring{X}$ via (\ref{Equation 4.1}), with a global trivializing section $d^n\theta = d\theta^1 \wedge \cdots \wedge d\theta^n$ which is flat with respect to the canonical flat $\overline{\mathcal{P}^{\operatorname{h}}}$-connection on $K^{\mathcal{P}^{\operatorname{h}}}$. We define the complex line bundle $\sqrt{\lvert K^{\mathcal{P}^{\operatorname{h}}} \rvert}$ over $\mathring{X}$ with flat $\overline{\mathcal{P}^{\operatorname{h}}}$-connection by declaring that $\sqrt{\lvert d^n \theta \rvert}$ is a flat global trivializing section of it. 
% the corresponding map $\sqrt{\lvert d^n \theta \rvert}: \mathcal{X}(\mathring{X})^n \to \mathcal{C}^\infty(\mathring{X})$
The flat $\overline{\mathcal{P}^{\operatorname{v}}}$-connection $\nabla^\hbar$ restricted on $\mathring{X}$, extends to an ordinary Hermitian connection on $(L_+^k \vert_{\mathring{X}}, h^{L_+^k} \vert_{\mathring{X}})$, which is denoted by the same symbol, with connection $1$-form $-\sqrt{-1} kx \cdot d\theta$ associated to the local trivializing section $\mathbf{1}_{\circ, L_+^k}^{\operatorname{U}(1)}$ and curvature $2$-form $-\sqrt{-1}k\omega$. 
% It cannot be extended to an ordinary Hermitian connection on the whole $(L_+^k, h^{L_+^k})$ as $-\sqrt{-1}(x_{v, \hbar} + \tfrac{\sqrt{-1}}{2} \mathbf{1}) \cdot d\theta_v$, where $\mathbf{1} = (1, ..., 1)$.
Note that $\tau \in \Gamma(\mathring{X}, L_+^k)$ is $\mathcal{P}^{\operatorname{h}}$-polarized, i.e. $\nabla_{\overline{\xi}}^\hbar \tau = 0$ for all $\xi \in \Gamma(\mathring{X}, \overline{\mathcal{P}^{\operatorname{h}}})$, if and only if $\tau_\circ$ is the pullback of a smooth function on $T$. The $k$-level half-form corrected prequantum line bundle in polarization $\mathcal{P}^{\operatorname{h}}$ is then the line bundle $L_+^k \vert_{\mathring{X}} \otimes \sqrt{\lvert K^{\mathcal{P}^{\operatorname{h}}} \rvert}$ over $\mathring{X}$.\par
Next we construct a pairing between $\Gamma(\mathring{X}, \overline{L}_+^k)' \sqrt{\lvert d^n x \rvert}$ and $\Gamma(\mathring{X}, L_+^k) \sqrt{\lvert d^n \theta \rvert}$. Since there is a natural pairing between $\Gamma(\mathring{X}, \overline{L}_+^k)'$ and $\Gamma(\mathring{X}, L_+^k)$, it suffices to know how to pair up $\sqrt{\lvert d^nx \rvert}$ and $\sqrt{\lvert d^n \theta \rvert}$. The construction of the BKS pairing often requires the indefinite pairing on the space of $n$-forms on each open subset $U$ of $X$ given by
\begin{equation}
\langle \alpha, \beta \rangle = \frac{n! \cdot \alpha \wedge \overline{\beta}}{(-1)^{\frac{n(n+1)}{2}} (2k\sqrt{-1} \omega \vert_U)^n} \in \mathcal{C}^\infty(U, \mathbb{C})
\end{equation}
for any $n$-forms $\alpha, \beta$ on $U$. With this motivation, we define
\begin{equation}
\label{Equation 5.02}
\langle \sqrt{\lvert \alpha \rvert}, \sqrt{\lvert \beta \rvert} \rangle = \sqrt{\left\lvert \frac{n! \cdot \alpha \wedge \overline{\beta}}{(-1)^{\frac{n(n+1)}{2}} (2k\sqrt{-1} \omega \vert_U)^n} \right\rvert} \in \mathcal{C}^\infty(U).
\end{equation}
Note that the formula (\ref{Equation 5.02}) is well-defined. Indeed, if $\alpha, \alpha'$ are $n$-forms on $U$ such that $\sqrt{\lvert \alpha \rvert} = \sqrt{\lvert \alpha' \rvert}$, then we can check that $\alpha = \eta \alpha'$ for some $\eta \in \mathcal{C}^\infty(U, \operatorname{U}(1))$ and therefore the definition of $\langle \sqrt{\lvert \alpha \rvert}, \sqrt{\lvert \beta \rvert} \rangle$ by (\ref{Equation 5.02}) coincides with that of $\langle \sqrt{\lvert \alpha' \rvert}, \sqrt{\lvert \beta \rvert} \rangle$. Similarly, if $\beta, \beta'$ are $n$-forms on $U$ such that $\sqrt{\lvert \beta \rvert} = \sqrt{\lvert \beta' \rvert}$, then the definition of $\langle \sqrt{\lvert \alpha \rvert}, \sqrt{\lvert \beta \rvert} \rangle$ coincides with that of $\langle \sqrt{\lvert \alpha \rvert}, \sqrt{\lvert \beta' \rvert} \rangle$. 
%Can it be extended to
%\begin{equation}
%\langle f\sqrt{\lvert \eta \rvert}, g\sqrt{\lvert \eta' \rvert} \rangle = f\overline{g} \langle \sqrt{\lvert \eta \rvert}, \sqrt{\lvert \eta' \rvert} \rangle?
%\end{equation}
% check the definition of half-density on $X$?
% https://www.mathematik.uni-muenchen.de/~schotten/GQ/12%20Metaplectic%20Correction.1.1.pdf
% https://pentagono.uniandes.edu.co/~acardona/GMQ.pdf
In particular, $\langle \sqrt{\lvert d^nx \rvert}, \sqrt{\lvert d^n\theta \rvert} \rangle = (2k)^{-\frac{n}{2}}$. We now have a sesquillinear pairing
\begin{equation*}
	\langle \quad, \quad \rangle: \Gamma(\mathring{X}, \overline{L}_+^k)' \sqrt{\lvert d^n x \rvert} \times \Gamma(\mathring{X}, L_+^k) \sqrt{\lvert d^n \theta \rvert} \to \mathbb{C}
\end{equation*}
defined by
\begin{equation}
	\langle s \sqrt{\lvert d^n x \rvert}, \tau \sqrt{\lvert d^n \theta \rvert} \rangle = \langle \langle \sqrt{\lvert d^nx \rvert}, \sqrt{\lvert d^n\theta \rvert} \rangle s, \tau \rangle = (2k)^{-\frac{n}{2}} \langle s, \tau \rangle,
\end{equation}
for all $s \in \Gamma(\mathring{X}, \overline{L}_+^k)'$ and $\tau \in \Gamma(\mathring{X}, L_+^k)$.\par
For any $m \in \mathbb{Z}^n$, define
% $\nabla_{\partial_{x^i}}^\hbar \tau = \left( \partial_{x^i} \tau_\circ \right) \mathbf{1}_{\circ, L_+^k}^{\operatorname{U}(1)}$.
\begin{equation}
\check{\sigma}_\hbar^m = \frac{k^{\frac{n}{2}}}{\pi^{\frac{n}{4}}} e^{\sqrt{-1} m \cdot \theta} \mathbf{1}_{\circ, L_+^k}^{\operatorname{U}(1)} \sqrt{\lvert d^n \theta \rvert} \in \Gamma \left( \mathring{X}, L_+^k \vert_{\mathring{X}} \sqrt{\lvert K^{\mathcal{P}^{\operatorname{h}}} \rvert} \right) \cong \Gamma(\mathring{X}, L_+^k) \sqrt{\lvert d^n\theta \rvert}.
\end{equation}
It is easy to check that $\langle \sigma_\hbar^m, \check{\sigma}_\hbar^{m'} \rangle = \delta_{m, m'}$ for all $m, m' \in \Lambda_\hbar$. We define the kernel
\begin{align*}
K^\hbar = \sum_{m \in \Lambda_\hbar} \check{\sigma}_\hbar^m \otimes \sigma_\hbar^m \in \check{\mathcal{H}}^\hbar \otimes_\mathbb{C} \mathcal{H}^\hbar,
\end{align*}
where $\check{\mathcal{H}}^\hbar := \bigoplus_{m \in \Lambda_\hbar} \mathbb{C} \check{\sigma}_\hbar^m$ is a subspace of $\Gamma(\mathring{X}, L_+^k) \sqrt{\lvert d^n \theta \rvert}$, 
%We define an inner product on $\check{\mathcal{H}}^\hbar$ by pulling back to $(0, \theta)$ and using half density $\sqrt{\lvert d^n \theta \rvert}$.
%\begin{align*}
%\langle s, s' \rangle = \int_T s(0, \theta) \overline{s'(0, \theta)},
%\end{align*}
%for $s, s' \in \check{\mathcal{H}}^\hbar$. Then $\{\check{\sigma}_\hbar^m\}_{m \in \Lambda_\hbar}$ forms a unitary basis of $\check{\mathcal{H}}^\hbar$.
and the $\mathbb{C}$-linear map
\begin{align*}
\Pi^\hbar: \Gamma(\mathring{X}, \overline{L}_+^k)' \sqrt{\lvert d^n x \rvert} \to \mathcal{H}^\hbar, \quad s \mapsto \Pi^\hbar s := (\langle \quad, \quad \rangle \otimes \operatorname{Id})(s \otimes K^\hbar),
\end{align*}
which is clearly a projection onto $\mathcal{H}^\hbar$.

\begin{definition}
	The \emph{Toeplitz operator for the pair of polarizations} $(\mathcal{P}^{\operatorname{h}}, \mathcal{P}^{\operatorname{v}})$ is the map
	\begin{align*}
	Q^\hbar: \mathcal{C}^\infty(X, \mathbb{C}) \times \mathcal{H}^\hbar \to \mathcal{H}^\hbar, \quad (f, s) \mapsto Q_f^\hbar s := \Pi^\hbar (fs).
	\end{align*}
\end{definition}

This definition keeps our promise in Section \ref{Section 3} that $Q_f^\hbar$ only depends on the restriction of $f$ on $\mathring{X}$, so that the operator $Q_f^\hbar \circ Q_g^\hbar - Q_{f \star_N g}^\hbar$ makes sense. We can write $Q^\hbar$ in terms of basis $\{\sigma_\hbar^m\}_{m \in \Lambda_\hbar}$ as follows.

\begin{proposition}
	\label{Proposition 3.9}
	Suppose $f \in \mathcal{C}^\infty(X, \mathbb{C})$. Then for any $m \in \Lambda_\hbar$,
	\begin{equation}
		\label{Equation 5.5}
		Q_f^\hbar \sigma_\hbar^m = \sum_{q \in \Lambda_\hbar - m} \widehat{f}_q(\hbar m) \sigma_\hbar^{m+q}.
	\end{equation}
	In general, for any $s = \sum_{m \in \Lambda_\hbar} s_m \sigma_\hbar^m \in \mathcal{H}^\hbar$,
	\begin{equation}
	\label{Equation 7.1}
	Q_f^\hbar s = \sum_{m \in \Lambda_\hbar} \left( \sum_{m' \in \Lambda_\hbar} \widehat{f}_{m-m'}(\hbar m') s_{m'} \right) \sigma_\hbar^m.
	\end{equation}
	% In other words, for all $p, q \in \Lambda_\hbar$, the $(p, q)$-entry of the matrix of $Q_f^\hbar$ relative to the basis $(\sigma_\hbar^m)_{m \in \Lambda_\hbar}$ is $\widehat{f}_{p-q}(\hbar q)$.
\end{proposition}
\begin{proof}
	For all $m, m' \in \Lambda_\hbar$,
	\begin{align*}
	\langle f \sigma_\hbar^{m'}, \check{\sigma}_\hbar^m \rangle = \frac{1}{(2\pi)^n} \int_T e^{-\sqrt{-1}(m-m') \cdot \theta} f(\hbar m', \theta) d^n\theta = \widehat{f}_{m-m'}(\hbar m').
	\end{align*}
	Therefore, we have (\ref{Equation 7.1}) and (\ref{Equation 5.5}) directly follows from it.
\end{proof}

\subsection{Locality}
\label{Subsection 5.2}
\quad\par
For $f, g \in \mathcal{C}^\infty(X, \mathbb{C})$ and $N \in \mathbb{N} \cup \{0\}$, define $f \star_N g = \sum_{i=0}^N \hbar^i C_i(f, g)$. Both the toric polarization $\mathcal{P}^{\operatorname{v}}$ and the toric star product $\star$ have singularity on $X \backslash \mathring{X}$. Because of this issue, we need certain treatments on the estimation of the operator norm of $Q_f^\hbar \circ Q_g^\hbar - Q_{f \star_N g}^\hbar$. In order to do so, we first need to discuss \emph{locality} in our quantization scheme.\par
On the one hand, we know that the toric star product $\star$ is local:  % $\operatorname{supp}(f \star g) \subset \operatorname{supp} f \cap \operatorname{supp} g$ for $f, g \in \mathcal{C}^\infty(X, \mathbb{C})[[\hbar]]$.
$(f \star g)(x, \theta)$ only depends on the jets of $f$ and $g$ at the point $(x, \theta) \in \mathring{X}$, so $\star$ induces for every open subset $U$ of $\mathring{X}$ a $\mathbb{C}[[\hbar]]$-bilinear map:
\begin{align*}
	\mathcal{C}^\infty(U, \mathbb{C})[[\hbar]] \otimes_{\mathbb{C}[[\hbar]]} \mathcal{C}^\infty(U, \mathbb{C})[[\hbar]] \to \mathcal{C}^\infty(U \cap \mathring{X}, \mathbb{C})[[\hbar]].
\end{align*}
On the other hand, motivated by the work of \cite{W2000} which studies locality in formal GNS representations in deformation quantization, we define a sheaf of modules $\underline{\mathcal{H}}^\hbar$ on $P_X$, whose space of global sections is $\mathcal{H}^\hbar$, as follows.

\begin{definition}
	The sheaf of modules $\underline{\mathcal{H}}^\hbar$ on $P_X$ consists of the following data.
	\begin{itemize}
		\item For any open subset $V$ of $P_X$, $\underline{\mathcal{H}}^\hbar(V)$ is defined to be the subspace $\mathcal{H}_V^\hbar$ of distributional sections in $\mathcal{H}^\hbar$ supported on $\mu^{-1}(V)$. Then $\mathcal{H}_V^\hbar = \sum_{m \in (kV) \cap \Lambda_\hbar} \mathbb{C} \sigma_\hbar^m$ is a Hilbert subspace of $\mathcal{H}^\hbar$. Also, $\mathcal{C}^\infty(V, \mathbb{C})$ acts on $\mathcal{H}_V^\hbar$ by $f \cdot s = (\mu^*f) s$.
		\item For all open subsets $V, V' \in P_X$ with $V \subset V'$, the restriction map $\underline{\mathcal{H}}^\hbar(V') \to \underline{\mathcal{H}}^\hbar(V)$ is defined to be the orthogonal projection $\pi_{V, V'}^\hbar: \mathcal{H}_{V'}^\hbar \to \mathcal{H}_V^\hbar$ (we write $\pi_V^\hbar$ in place of $\pi_{P_X, V}^\hbar$ for simplicity).
	\end{itemize}
\end{definition}

A natural question is whether our construction of Toeplitz operators $Q_f^\hbar$ are local in an appropriate sense. The following definition is analogous to Definition 3.5 in \cite{W2000}.

\begin{definition}
	A linear operator $A$ on $\mathcal{H}^\hbar$ is called \emph{local} if for any $s \in \mathcal{H}^\hbar$, $\operatorname{supp} (As) \subset \operatorname{supp} s$.
\end{definition}

In general, $Q_f^\hbar$ might not be a local operator on $\mathcal{H}^\hbar$ in this sense - for distinct $m, m' \in \Lambda_\hbar$, we can always find a function $f \in \mathcal{C}^\infty(X, \mathbb{C})$ such that $Q_f^\hbar \sigma_\hbar^m = \sigma_\hbar^{m'}$ and hence $\operatorname{supp} Q_f^\hbar \sigma_\hbar^m \not\subset \operatorname{supp} \sigma_\hbar^m$. In other words, $Q_f^\hbar$ might not preserve subspaces $\mathcal{H}_V^\hbar$'s.\par
% need to make a good notion: half-local? polarized local? fibrewise local?
However, the sequence of Toeplitz operators $\{ Q_f^\hbar \}_{k \in 2\mathbb{N} - 1, k \geq k_0}$ is `local' in $f$. By \emph{locality} we mean that for any open subset $V$ of $P_X$, the restriction $f \vert_{\mu^{-1}(V)}$ and the sequence $\left\{ \left. Q_f^\hbar \right\vert_{\mathcal{H}_V^\hbar} \right\}_{k \in 2\mathbb{N} - 1, k \geq k_0}$ of restricted Toeplitz operators are completely determined by each other. The precise statement is given as follows.

\begin{proposition}
	\label{Proposition 3.11}
	Let $f \in \mathcal{C}^\infty(X, \mathbb{C})$ and $V$ be an open subset of $P_X$. Then the following statements hold.
	\begin{enumerate}
		\item If $f \vert_{\mu^{-1}(V)} = 0$, then $Q_f^\hbar \vert_{\mathcal{H}_V^\hbar} = 0$ for all odd $k \in \mathbb{N}$ with $k \geq k_0$.
		\item If there exists $\delta \in (0, \tfrac{1}{k_0})$ such that $Q_f^\hbar \vert_{\mathcal{H}_V^\hbar} = 0$ for all odd $k \in \mathbb{N}$ with $\hbar = \tfrac{1}{k} < \delta$, then $f \vert_{\mu^{-1}(V)} = 0$.
	\end{enumerate}
\end{proposition}

We need a lemma for its proof.

\begin{lemma}
	\label{Lemma 3.12}
	Let $K$ be a compact subset of $\mathring{P}_X$ and $Z$ be a finite subset of $\mathbb{Z}^n$. Then there exists $\delta \in (0, \tfrac{1}{k_0})$ such that for all odd $k \in \mathbb{N}$ with $\hbar = \tfrac{1}{k} < \delta$, if $r \in (k K) \cap \Lambda_\hbar$ and $q \in Z$, then
	\begin{align*}
	r + q \in \Lambda_\hbar.
	\end{align*}
\end{lemma}
\begin{proof}
	Since $K$ is compact and $\Sigma^{(1)}$ is finite, there exists $\varepsilon \in (0, \tfrac{1}{k_0})$ such that for all $F \in \Sigma^{(1)}$ and $x \in K$, $l_F(x) \geq \varepsilon$. Also, as $Z$ is finite, there exists $\delta \in (0, \tfrac{\varepsilon}{2})$ such that for all $F \in \Sigma^{(1)}$ and $q \in Z$,
	\begin{equation}
	\label{Equation 3.5}
	\delta \lvert \langle q, \nu_F \rangle \rvert \leq \tfrac{\varepsilon}{2}.
	\end{equation}
	Consider any odd $k \in \mathbb{N}$ with $\hbar = \tfrac{1}{k} < \delta$. Suppose $r \in (k K) \cap \Lambda_\hbar$ and $q \in Z$. Then $\hbar r \in K$. For all $F \in \Sigma^{(1)}$, by (\ref{Equation 3.5}),
	\begin{align*}
	l_F(\hbar(r + q)) = & l_F(\hbar r) + \hbar \langle q, \nu_F \rangle \geq l_F(\hbar r) - \hbar \lvert \langle q, \nu_F \rangle \rvert \geq \varepsilon - \delta \lvert \langle q, \nu_F \rangle \rvert \geq \tfrac{\varepsilon}{2},\\
	\therefore l_F^{L_+^k} (r + q) = & l_{F, \hbar}(r + q) - \tfrac{1}{2} = k l_F(\hbar(r + q)) - \tfrac{1}{2} \geq \tfrac{1}{2} (k \varepsilon - 1) > \tfrac{1}{2} (\tfrac{\varepsilon}{\delta} - 1) > 0.
	\end{align*}
	It implies that $r + q \in P_{L_+^k}$. Finally, because $r \in \Lambda_\hbar$ and $q \in \mathbb{Z}^n$, $r + q \in \Lambda_\hbar$.
\end{proof}

\begin{proof}[\myproof{Proposition}{\ref{Proposition 3.11}}]
	\quad
	\begin{enumerate}
		\item Fix odd $k \in \mathbb{N}$ with $k \geq k_0$. We observe from (\ref{Equation 7.1}) that if $s \in \mathcal{H}^\hbar$ and $\operatorname{supp} f \cap \operatorname{supp} s = \emptyset$, then $Q_f^\hbar s = 0$. Then it directly follows that $Q_f^\hbar \vert_{\mathcal{H}_V^\hbar} = 0$.
		% If $f \vert_{\operatorname{supp} s} = 0$, then $Q_f^\hbar s = 0$.
		\item Note that $\mathcal{H}_V^\hbar = \mathcal{H}_{V \cap \mathring{P}_X}^\hbar$ for any odd $k \in \mathbb{N}$ with $k \geq k_0$. Since $\mu^{-1}(V \cap \mathring{P}_X)$ is dense in $\mu^{-1}(V)$, $f \vert_{\mu^{-1}(V \cap \mathring{P}_X)} = 0$ implies $f \vert_{\mu^{-1}(V)} = 0$. Thus, without loss of generality, we assume $V \subset \mathring{P}_X$. It is then clear that $f \vert_{\mu^{-1}(V)} = 0$ if and only if for all $p \in \mathbb{Z}^n$, $\widehat{f}_p \vert_V = 0$. Now fix $p \in \mathbb{Z}^n$. It remains to show that $\widehat{f}_p \vert_V = 0$.\par
		Consider first the special case when $V$ is a relatively compact open subset of $\mathring{P}_X$. Then the closure $\overline{V}$ of $V$ in $\mathring{P}_X$ is compact. By Lemma \ref{Lemma 3.12}, we can choose $\delta' \in (0, \delta)$ such that for all odd $k \in \mathbb{N}$ with $\hbar = \tfrac{1}{k} < \delta'$, if $r \in (k V) \cap \Lambda_\hbar \subset (k \overline{V}) \cap \Lambda_\hbar$, then $r + p \in \Lambda_\hbar$ and therefore $\widehat{f}_p(\hbar r) = 0$ as it follows from (\ref{Equation 7.1}) that the value $\widehat{f}_p(\hbar r)$ appears as an entry of the matrix form of the restricted operator $Q_f^\hbar \vert_{\mathcal{H}_V^\hbar}: \mathcal{H}_V^\hbar \to \mathcal{H}^\hbar$. We have shown that $\widehat{f}_p$ vanishes on the dense subset
		\begin{align*}
		V \cap \bigcup_{\substack{k \in 2\mathbb{N} - 1 \\ \hbar = 1/k < \delta'}} (\hbar \Lambda_\hbar)
		\end{align*}
		of $V$, therefore $\widehat{f}_p \vert_V = 0$.\par
		For the general case, we pick an open cover $\{V_i\}_{i \in I}$ of $V$ such that for all $i \in I$, $V_i$ is a relatively compact open subset of $\mathring{P}_X$ and apply the above argument for each $V_i$. Then we have the conclusion that $\widehat{f}_p \vert_V = 0$. % $V_i$ can be taken as open balls of small enough radii.
	\end{enumerate}
\end{proof}

%Borrowing the idea in \cite{W2000}, for a $\mathbb{C}$-linear functional $\varphi: \mathcal{C}^\infty(X, \mathbb{C}) \to \mathbb{C}$, define the \emph{support} of $\varphi$, denoted by $\operatorname{supp} \varphi$, as
%\begin{align*}
%\operatorname{supp} \varphi := X \backslash \bigcup \{ U: U \text{ is an open subset of } X \text{ and } \varphi \vert_{\mathcal{C}_{\operatorname{c}}^\infty(U, \mathbb{C})} = 0 \}.
%\end{align*}

%For $s \in \mathcal{H}^\hbar$, we can define a $\mathbb{C}$-linear functional $\varphi_s: \mathcal{C}^\infty(X, \mathbb{C}) \to \mathbb{C}$ by
%\begin{align*}
%\varphi_s(f) = \langle Q_f^\hbar s, s \rangle.
%\end{align*}
%It seems that $\operatorname{supp} s = \operatorname{supp} \varphi_s$.

\subsection{Norm estimations}
\label{Subsection 5.3}
\quad\par
We first define the following norms
\begin{align*}
\lVert s \rVert_1 = \sum_{m \in \Lambda_\hbar} \lvert s_m \rvert \quad \text{and} \quad \lVert s \rVert_\infty = \sup_{m \in \Lambda_\hbar} \lvert s_m \rvert,
\end{align*}
for any $s = \sum_{m \in \Lambda_\hbar} s_m \sigma_\hbar^m \in \mathcal{H}^\hbar$, and denote by $\lVert A \rVert_1$ and $\lVert A \rVert_\infty$ the operator norm of an operator $A$ on $\mathcal{H}^\hbar$ with respect to $\lVert \quad \rVert_1$ and $\lVert \quad \rVert_\infty$ respectively. We then have the following variant of H\"older inequality (one might consult Chapter 5 in \cite{HJ2013}; see also Lemma 6.3 in \cite{LY2021}). % and see Problem 5.6.P21 there).

\begin{lemma}
	\label{Lemma 6.5}
	Let $r \in \mathbb{N}$. For $1 \leq p \leq \infty$, let $\lVert \quad \rVert_p$ be the operator norm with respect to the $l^p$ norm on $\mathbb{C}^r$. Suppose $1 \leq p, q \leq \infty$ are H\"older conjugates, i.e. $\tfrac{1}{p} + \tfrac{1}{q} = 1$ (by convention, when $p = 1$, $q = \infty$). Then for any $r \times r$ complex matrix $A$,
	\begin{equation*}
	\lVert A \rVert_2^2 \leq \lVert A \rVert_p \lVert A \rVert_q.
	\end{equation*}
\end{lemma}

\begin{proposition}
	For all $f \in \mathcal{C}^\infty(X, \mathbb{C})$, there exists $K > 0$ such that
	\begin{align*}
	\left\lVert Q_f^\hbar \right\rVert \leq K
	\end{align*}
	for all odd $k \in \mathbb{N}$ with $k \geq k_0$.
\end{proposition}
\begin{proof}
	By Lemma \ref{Lemma 5.3}, there exists $K > 0$ such that for all map $\mathbb{Z}^n \to \mathring{P}_X$ given by $m \mapsto x_m$, $\sum_{m \in \mathbb{Z}^n} \lvert \widehat{f}_m(x_m) \rvert \leq K$. Let $k \in \mathbb{N}$ be odd and $k \geq k_0$. For any $s \in \mathcal{H}^\hbar$, writing $s = \sum_{m \in \Lambda_\hbar} s_m \sigma_\hbar^m$,
	\begin{align*}
	\left\lVert Q_f^\hbar s \right\lVert_1 = & \sum_{m \in \Lambda_\hbar} \left\lvert \sum_{m' \in \Lambda_\hbar} \widehat{f}_{m-m'}(\hbar m') s_{m'} \right\rvert \leq K \lVert s \rVert_1.\\
	\left\lVert Q_f^\hbar s \right\lVert_\infty = & \sup_{m \in \Lambda_\hbar} \left\lvert \sum_{m' \in \Lambda_\hbar} \widehat{f}_{m-m'}(\hbar m') s_{m'} \right\rvert \leq K \lVert s \rVert_\infty.
	\end{align*}
	Thus, $\lVert Q_f^\hbar \rVert_1 \leq K$ and $\lVert Q_f^\hbar \rVert_\infty \leq K$. By Lemma \ref{Lemma 6.5}, $\lVert Q_f^\hbar \rVert \leq \sqrt{\lVert Q_f^\hbar \rVert_1 \lVert Q_f^\hbar \rVert_\infty} \leq K$.
\end{proof}

Therefore, the operator norm $\lVert Q_f^\hbar \rVert$ of $Q_f^\hbar$ has a uniform bound, independent of $\hbar$.\par
Now as mentioned in Subsection \ref{Subsection 5.2}, we need careful treatments on $\lVert Q_f^\hbar \circ Q_g^\hbar - Q_{f \star_N g}^\hbar \rVert$, because of \emph{`truncated'} shift operators. Let us first explain this in the example $X = \mathbb{C}\mathbb{P}^1$. Suppose $f, g \in \mathcal{C}^\infty(X, \mathbb{C})$ are of the form
\begin{align*}
f(x, \theta) = \widehat{f}_p(x) e^{\sqrt{-1}p \cdot \theta}, \quad g(x, \theta) = \widehat{g}_q(x) e^{\sqrt{-1}q \cdot \theta} \quad \text{on } \mathring{X},
\end{align*}
where $p, q \in \mathbb{Z}$. Pick the quantum state $\sigma_\hbar^m$ supported on the Bohr-Sommerfeld (BS) fibre over the point $x = \hbar m$ with $m \in \Lambda_\hbar$. On the one hand, $Q_g^\hbar$ shifts the support of $\sigma_\hbar^m$ by $\hbar q$ to the BS fibre over the position $x = \hbar (m + q)$ and then $Q_f^\hbar$ shifts that of $Q_g^\hbar \sigma_\hbar^m$ by $\hbar p$ to the BS fibre over $x = \hbar (m + p + q)$. On the other hand, $Q_{f \star g}^\hbar$ directly shifts the support of $\sigma_\hbar^m$ to the BS fibre over the final position $x = \hbar (m + p + q)$.\par
It might happen that the final position $x = \hbar (m + p + q)$ is within the polytope $P_X$ but the intermediate position $x = \hbar (m + q)$ lies outside $P_X$. This happens quite often when either the initial position $x = \hbar m$ is too close to the boundary of $P_X$ or $\hbar q$ is too large, resulting in the vanishing of $(Q_f^\hbar \circ Q_g^\hbar) \sigma_\hbar^m$ and the survival of $Q_{f \star g}^\hbar \sigma_\hbar^m$. Unlike the case of symplectic tori \cite{LY2021}, we might fail to obtain the approximation $Q_{f \star g}^\hbar \sigma_\hbar^m \approx \widehat{f}_p(m + \hbar q) \widehat{g}_q(m) \sigma_\hbar^{m+p+q}$ by Taylor series as $\widehat{f}_p$ might not extend smoothly by zero on $\mathbb{R}$. We picturize the situation as follows.
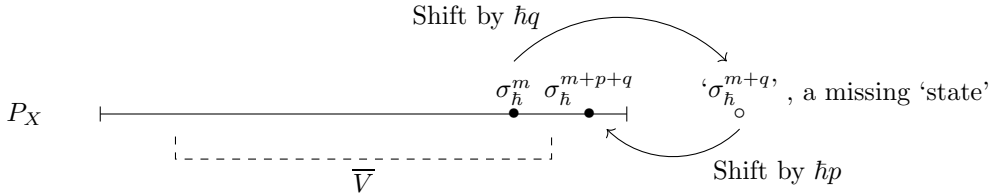
\begin{center}
	\begin{tikzpicture}
	\node at (-4.5, 0) {$P_X$};
	\draw (-3.5, 0) -- (3.5, 0);
	\draw (-3.5, -.1) -- (-3.5, .1);
	\draw (3.5, -.1) -- (3.5, .1);
	
	\node at (0, -.6) [below] {$\overline{V}$};
	\draw[dashed] (-2.5, -.3) -- (-2.5, -.6) -- (2.5, -.6) -- (2.5, -.3);
	
	\node at (2, 0) [above] {$\sigma_\hbar^m$};
	\node at (2, 0) {\small $\bullet$};
	\node at (3, 0) [above] {$\sigma_\hbar^{m+p+q}$};
	\node at (3, 0) {\small $\bullet$};
	\node at (5, 0) [above] {`$\sigma_\hbar^{m+q}$'};
	\node at (5, 0) {\small $\circ$};
	\node at (5.5, 0) [above right] {, a missing `state'};
	
	\node at (1.5, 1) [above] {Shift by $\hbar q$};
	\draw[->] (2, .7) arc (135: 45: 2);
	\node at (5.5, -.5) [below] {Shift by $\hbar p$};
	\draw[->] (5, -.2) arc (-45: -135: 1.25);
	\end{tikzpicture}
\end{center}\par
To avoid this situation, we cut off quantum states too close to the boundary of $P_X$ by taking a relatively compact open subset $V$ of $\mathring{P}_X$. Still, we need $\hbar q$ small. Note that $\hbar, q$ are competing factors - $\hbar$ is considered small but in general, $g$ can have non-zero fibrewise Fourier coefficients $\widehat{g}_q$ for $q$ as large as we want. Therefore, we need $g \in \mathcal{C}^\infty(X, \mathbb{C})_{<\infty}$ so as to set a bound controlling $q$. Eventually, we examine the phenomenon as $\hbar$ tends to zero. After all, quantization is in principle a process passing from classical physics to quantum physics when $\hbar$ is infinitesimally small.\par
For an open subset $V$ of $P_X$ and an operator $A$ on $\mathcal{H}^\hbar$, we denote by $\lVert A \rVert_V$ the operator norm of the restriction $A \vert_{\mathcal{H}_V^\hbar}: \mathcal{H}_V^\hbar \to \mathcal{H}^\hbar$ with respect to the inner products, i.e.
\begin{align*}
\lVert A \rVert_V := \sup_{s \in \mathcal{H}_V^\hbar, s \neq 0} \frac{\lVert A(s) \rVert}{\lVert s \rVert}.
\end{align*}
Note that
\begin{equation*}
\lVert A \rVert_V = \sup_{s \in \mathcal{H}_V^\hbar, s \neq 0} \frac{\lVert (A \circ \pi_V^\hbar)(s) \rVert}{\lVert s \rVert} \leq \sup_{s \in \mathcal{H}^\hbar, s \neq 0} \frac{\lVert (A \circ \pi_V^\hbar)(s) \rVert}{\lVert s \rVert} = \lVert A \circ \pi_V^\hbar \rVert \leq \lVert \pi_V^\hbar \rVert \lVert A \rVert_V \leq \lVert A \rVert_V.
\end{equation*}
Therefore, $\lVert A \rVert_V = \lVert A \circ \pi_V^\hbar \rVert$. As explained in Proposition \ref{Proposition 3.11}, the sequence of Toeplitz operators $Q_f^\hbar$ is local in $f$. We shall estimate the operator norm of the restriction
\begin{align*}
\left. \left( Q_f^\hbar \circ Q_g^\hbar - Q_{f \star_N g}^\hbar \right) \right\vert_{\mathcal{H}_V^\hbar}: \mathcal{H}_V^\hbar \to \mathcal{H}^\hbar
\end{align*}
as $\hbar \to 0^+$. The upshot is the following theorem.

\begin{theorem}
	\label{Theorem 3.13}
	($=$ Theorem \ref{Theorem 1.1}) Let $f \in \mathcal{C}^\infty(X, \mathbb{C})$, $g \in \mathcal{C}^\infty(X, \mathbb{C})_{< \infty}$ and $V$ be a relatively compact open subset of $\mathring{P}_X$. Then there exists $\delta > 0$ such that for all $N \in \mathbb{N} \cup \{0\}$, there exists $K > 0$ such that
	\begin{equation*}
	\left\lVert Q_f^\hbar \circ Q_g^\hbar - Q_{f \star_N g}^\hbar \right\rVert_V \leq K \hbar^{N+1},
	\end{equation*}
	for all odd $k \in \mathbb{N}$ with $\hbar = \tfrac{1}{k} < \delta$.
\end{theorem}
\begin{proof} %[\myproof{Theorem}{\ref{Theorem 3.13}}]
	% Fix $f, g \in \mathcal{C}^\infty(X, \mathbb{C})_{<\infty}$ and a relatively compact open subset $V$ of $\mathring{P}_X$. Then the closure $\overline{V}$ of $V$ in $\mathring{P}_X$ is a compact subset of $\mathring{P}_X$.
	By assumption, there exists $l \in \mathbb{N} \cup \{0\}$ such that $g \in \mathcal{C}^\infty(X, \mathbb{C})_{\leq l}$. Define the finite subset
	\begin{align*}
	Z = \{ (q_1, ..., q_n) \in \mathbb{Z}^n: \max \{\lvert q_1 \rvert, ..., \lvert q_n \rvert\} \leq l \}
	\end{align*}
	of $\mathbb{Z}^n$. By Lemma \ref{Lemma 3.12}, we pick $\delta \in (0, \tfrac{1}{k_0})$ so that if $k \in \mathbb{N}$ is odd with $\hbar = \tfrac{1}{k} < \delta$, $r \in (k \overline{V}) \cap \Lambda_\hbar$ and $q \in Z$, then $r + q \in \Lambda_\hbar$.\par
	Fix $N \in \mathbb{N} \cup \{0\}$. We shall construct a bound $K$ by estimating the decay rate of the fibrewise Fourier coefficients of functions $f, g$ and their derivatives on $\mathring{X}$. For $m \in \mathbb{Z}^n$, define $[m]$ as in (\ref{Equation 5.2}). By Lemma \ref{Lemma 5.3}, for all multi-index $I$, there exist $C_{f, V, I}, C_{g, I} > 0$ such that for all $m \in \mathbb{Z}^n$,
	\begin{align*}
	\left\lvert \widehat{f}_m^{(I)}(x) \right\vert \leq & C_{f, V, I} [m]^{-2} \quad \text{for all } x \in \overline{V};\\
	\left\lvert m^I \widehat{g}_m(x) \right\vert \leq & C_{g, I} [m]^{-2} \quad \text{for all } x \in \mathring{P}_X.
	\end{align*}
	We see that $S := \sum_{m \in \mathbb{Z}^n} [m]^{-2} < +\infty$. Define
	\begin{align*}
	K = S^2 \sum_{\lvert I \rvert = N+1} \frac{C_{f, V, I} C_{g, I}}{I!} > 0.
	\end{align*}
	Consider any odd $k \in \mathbb{N}$ with $\hbar = \tfrac{1}{k} < \delta$. Define the error term $E_N^\hbar := Q_f^\hbar \circ Q_g^\hbar - Q_{f \star_N g}^\hbar$. Fix a quantum state $s = \sum_{m \in \Lambda_\hbar} s_m \sigma_\hbar^m \in \mathcal{H}^\hbar$.\par
	We shall apply Proposition \ref{Proposition 3.9} to compute $E_N^\hbar(s)$. On the one hand, we see that
	\begin{equation}
	\label{Equation 3.6}
	\begin{split}
	Q_f^\hbar Q_g^\hbar s = & \sum_{m \in \Lambda_\hbar} \left( \sum_{p, r \in \Lambda_\hbar} \widehat{f}_{m-p}(\hbar p) \widehat{g}_{p-r}(\hbar r) s_r \right) \sigma_\hbar^m\\
	= & \sum_{m \in \Lambda_\hbar} \left( \sum_{r \in \Lambda_\hbar} \sum_{q \in \Lambda_\hbar - r} \widehat{f}_{m-r-q}(\hbar (r+q)) \widehat{g}_q(\hbar r) s_r \right) \sigma_\hbar^m.
	\end{split}
	\end{equation}
	In the second line of (\ref{Equation 3.6}), we make a change of variable $q = p - r$. On the other hand, for any $i \in \mathbb{N} \cup \{0\}$ and $r \in \mathbb{Z}^n$, the  $r$th fibrewise Fourier coefficient of $C_i(f, g)$ is
	\begin{equation*}
	\sum_{\lvert I \rvert = i} \frac{1}{I!} \sum_{q \in \mathbb{Z}^n} q^I \cdot \widehat{f}_{r-q}^{(I)}(x) \widehat{g}_q(x),
	\end{equation*}
	% on $\mathring{X}$,
	% \begin{equation*}
	% C_i(f, g)(x, \theta) = \sum_{\lvert I \rvert = i} \frac{1}{I!} \sum_{r \in \mathbb{Z}^n} e^{\sqrt{-1} r \cdot \theta} \sum_{q \in \mathbb{Z}^n} q^I \cdot \widehat{f}_{r-q}^{(I)}(x) \widehat{g}_q(x),
	% \end{equation*}
	and hence
	\begin{equation}
	\label{Equation 3.7}
	Q_{C_i(f, g)}^\hbar s = \sum_{m \in \Lambda_\hbar} \left( \sum_{r \in \Lambda_\hbar} \sum_{q \in \mathbb{Z}^n} \sum_{\lvert I \rvert = i} \frac{q^I}{I!} \widehat{f}_{m-r-q}^{(I)}(\hbar r) \widehat{g}_q(\hbar r) s_r \right) \sigma_\hbar^m.
	\end{equation}
	Note that fixing $m, r \in \Lambda_\hbar$, the variable $q$ in the last line of (\ref{Equation 3.6}) only ranges over $\Lambda_\hbar - r$ while the same variable $q$ in (\ref{Equation 3.7}) ranges over $\mathbb{Z}^n$. When $q \not\in \Lambda_\hbar - r$, i.e. $r + q \not\in \Lambda_\hbar$ (we see that $m - r - q \neq (0, ..., 0)$ since $m \in \Lambda_\hbar$), the point $\hbar (r + q)$ is no longer in the moment polytope $P_X$ and hence $\widehat{f}_{m-r-q}(\hbar (r + q))$ is undefined. In layman's terms, large fibrewise Fourier modes of $g$ might be truncated when we apply the operator $Q_g^\hbar$ onto distributional sections supported on Bohr-Sommerfeld fibres close to $X \backslash \mathring{X}$, whereas some of these informations would possibly be recovered when we apply $Q_{f \star_N g}^\hbar$ on them.\par
	For the ease of expression, we extend every non-zero fibrewise Fourier coefficients $\widehat{f}_p$'s onto the entire $\mathfrak{t}^*$ by zero. Note that these might not be smooth extensions. Then one can write
	\begin{align*}
	E_N^\hbar(s) = \sum_{m \in \Lambda_\hbar} \left( \sum_{r \in \Lambda_\hbar} \sum_{q \in \mathbb{Z}^n} R_{m, r, q}^N \widehat{g}_q(\hbar r) s_r \right) \sigma_\hbar^m,
	\end{align*}
	where for all $m, r \in \Lambda_\hbar$, and $q \in \mathbb{Z}^n$, $R_{m, r, q}^N$ is the remainder term given by
	\begin{align*}
	R_{m, r, q}^N := \widehat{f}_{m-r-q}(\hbar (r + q)) - \sum_{\lvert I \rvert \leq N} \frac{(\hbar q)^I}{I!} \widehat{f}_{m-r-q}^{(I)}(\hbar r).
	\end{align*}
	If $r + q \in \Lambda_\hbar$, then by Taylor's Theorem, $R_{m, r, q}^N$ can be expressed in integral form:
	\begin{align*}
	R_{m, r, q}^N = (N+1) \sum_{\lvert I \rvert = N+1} \frac{\hbar^{N+1} q^I}{I!} \int_0^1 (1 - t)^N \widehat{f}_{m-r-q}^{(I)}(\hbar (r + t q)) dt.
	\end{align*}
	Otherwise, as we are not able to apply Taylor's Theorem, those terms $R_{m, r, q}^N \widehat{g}_q(\hbar r) s_r$ might not be controllable.\par
	The resort to continue our argument is to cut off these undesirable terms by composing with the orthogonal projection $\pi_V^\hbar: \mathcal{H}^\hbar \to \mathcal{H}_V^\hbar$. Since $\hbar < \delta$, we have the expression
	\begin{align*}
	(E_N^\hbar \circ \pi_V^\hbar)(s) = \sum_{m \in \Lambda_\hbar} \left( \sum_{r \in V_\hbar} \sum_{q \in \Lambda_\hbar - r} R_{m, r, q}^N \widehat{g}_q(\hbar r) s_r \right) \sigma_\hbar^m,
	\end{align*}
	where $V_\hbar = (k V) \cap \Lambda_\hbar$. Thus we have
	\begin{align*}
	\lVert (E_N^\hbar \circ \pi_V^\hbar)(s) \rVert_1 = & \hbar^{N+1} \sum_{r \in V_\hbar} \lvert s_r \rvert \sum_{q \in \Lambda_\hbar - r} \sum_{m \in \Lambda_\hbar} \sum_{\lvert I \rvert = N+1} \frac{C_{f, V, I} C_{g, I}}{I! \cdot [m-r-q]^2 [q]^2} \leq K \hbar^{N+1} \lVert s \rVert_1,\\
	\lVert (E_N^\hbar \circ \pi_V^\hbar)(s) \rVert_\infty = & \hbar^{N+1} \lVert s \rVert_\infty \sup_{m \in \Lambda_\hbar} \sum_{r \in V_\hbar} \sum_{q \in \Lambda_\hbar - r} \sum_{\lvert I \rvert = N+1} \frac{C_{f, V, I} C_{g, I}}{I! \cdot [m-r-q]^2 [q]^2} \leq K \hbar^{N+1} \lVert s \rVert_\infty.
	\end{align*}
	Thus, $\lVert E_N^\hbar \circ \pi_V^\hbar \rVert_1 \leq K \hbar^{N+1}$ and $\lVert E_N^\hbar \circ \pi_V^\hbar \rVert_\infty \leq K \hbar^{N+1}$. By Lemma \ref{Lemma 6.5}, we conclude that $\lVert E_N^\hbar \rVert_V = \lVert E_N^\hbar \circ \pi_V^\hbar \rVert \leq \sqrt{\lVert E_N^\hbar \circ \pi_V^\hbar \rVert_1 \lVert E_N^\hbar \circ \pi_V^\hbar \rVert_\infty} \leq K \hbar^{N+1}$.
\end{proof}

\appendix

%\subjclass[]{}
\keywords{}

\bibliographystyle{amsplain}
\bibliography{References}

\end{document}